\documentclass[11pt]{amsart}
\usepackage{graphicx}
\usepackage{amsmath,amsthm,amssymb,enumerate}
\usepackage{euscript,mathrsfs}
\usepackage[citecolor=blue,colorlinks=true]{hyperref}
\usepackage[left=3cm,right=3cm,top=4.5cm,bottom=4.5cm]{geometry}
\usepackage{color}
\catcode`\@=11 \@addtoreset{equation}{section}

\catcode`\@=12

\allowdisplaybreaks

\newtheorem{Theorem}{Theorem}[section]
\newtheorem{Proposition}[Theorem]{Proposition}
\newtheorem{Lemma}[Theorem]{Lemma}
\newtheorem{Corollary}[Theorem]{Corollary}

\theoremstyle{definition}
\newtheorem{Definition}[Theorem]{Definition}

\newtheorem{Remark}[Theorem]{Remark}

\newcommand{\bTheorem}[1]{
\begin{Theorem} \label{T#1} }
\newcommand{\eT}{\end{Theorem}}

\newcommand{\bProposition}[1]{
\begin{Proposition} \label{P#1}}
\newcommand{\eP}{\end{Proposition}}

\newcommand{\bLemma}[1]{
\begin{Lemma} \label{L#1} }
\newcommand{\eL}{\end{Lemma}}

\newcommand{\bCorollary}[1]{
\begin{Corollary} \label{C#1} }
\newcommand{\eC}{\end{Corollary}}

\newcommand{\bRemark}[1]{
\begin{Remark} \label{R#1} }
\newcommand{\eR}{\end{Remark}}
\newcommand{\E}{\mathcal{E}}
\newcommand{\bDefinition}[1]{
\begin{Definition} \label{D#1} }
\newcommand{\eD}{\end{Definition}}

\newcommand{\Q}{\Omega}

\newcommand{\dif}{\mathrm{d}}
\newcommand{\Del}{\Delta_x}

\newcommand{\vme}{\vm_\ep}

\newcommand{\rmk}[1]{\textcolor{red}{#1}}

\newcommand{\tvm}{\tilde{\vc{m}}}
\newcommand{\bfphi}{\boldsymbol{\varphi}}

\newcommand{\BV}{\mathrm{BV}}

\newcommand{\bFormula}[1]{
\begin{equation} \label{#1}}
\newcommand{\eF}{\end{equation}}

\newcommand{\Ov}[1]{\overline{#1}}

\newcommand{\DC}{C^\infty_c}

\newcommand{\aleq}{\lesssim}
\newcommand{\ageq}{\gtrsim}

\newcommand{\vr}{\varrho}
\newcommand{\vre}{\vr_\ep}

\newcommand{\vue}{\vu_\ep}
\newcommand{\tvr}{\tilde \vr}

\newcommand{\vt}{\vartheta}
\newcommand{\vu}{\vc{u}}
\newcommand{\vm}{\vc{m}}

\newcommand{\vc}[1]{{\bf #1}}

\newcommand{\Div}{{\rm div}_x}
\newcommand{\Grad}{\nabla_x}

\newcommand{\dx}{\,{\rm d} {x}}

\newcommand{\dt}{\,{\rm d} t }

\newcommand{\intO}[1]{\int_{\Omega} #1 \, \dx}

\newcommand{\intQ}[1]{\int_{{\Q}} #1 \ \dx}

\newcommand{\D}{{\rm d}}
\newcommand{\ep}{\varepsilon}

\newcommand{\mathd}{\mathrm{d}}
\newcommand{\tmop}{\text}

\definecolor{Cgrey}{rgb}{0.85,0.85,0.85}
\definecolor{Cblue}{rgb}{0.50,0.85,0.85}
\definecolor{Cred}{rgb}{1,0,0}
\definecolor{fancy}{rgb}{0.10,0.85,0.10}

\allowdisplaybreaks

\newcommand\Cbox[2]{%
    \newbox\contentbox%
    \newbox\bkgdbox%
    \setbox\contentbox\hbox to \hsize{%
        \vtop{
            \kern\columnsep
            \hbox to \hsize{%
                \kern\columnsep%
                \advance\hsize by -2\columnsep%
                \setlength{\textwidth}{\hsize}%
                \vbox{
                    \parskip=\baselineskip
                    \parindent=0bp
                    #2
                }%
                \kern\columnsep%
            }%
            \kern\columnsep%
        }%
    }%
    \setbox\bkgdbox\vbox{
        \color{#1}
        \hrule width  \wd\contentbox %
               height \ht\contentbox %
               depth  \dp\contentbox
        \color{black}
    }%
    \wd\bkgdbox=0bp%
    \vbox{\hbox to \hsize{\box\bkgdbox\box\contentbox}}%
    \vskip\baselineskip%
}

\date{}

\begin{document}


\title{Dissipative solutions and semiflow selection for the complete Euler system}

\author{Dominic Breit}
\address[D. Breit]{Department of Mathematics, Heriot-Watt University, Riccarton Edinburgh EH14 4AS, UK}
\email{d.breit@hw.ac.uk}

\author{Eduard Feireisl}
\address[E.Feireisl]{Institute of Mathematics AS CR, \v{Z}itn\'a 25, 115 67 Praha 1, Czech Republic
\and Institute of Mathematics, TU Berlin, Strasse des 17.Juni, Berlin, Germany }
\email{feireisl@math.cas.cz}
\thanks{The research of E.F. leading to these results has received funding from
the Czech Sciences Foundation (GA\v CR), Grant Agreement
18--05974S. The Institute of Mathematics of the Academy of Sciences of
the Czech Republic is supported by RVO:67985840.}

\author{Martina Hofmanov\'a}
\address[M. Hofmanov\'a]{Fakult\"at f\"ur Mathematik, Universit\"at Bielefeld, D-33501 Bielefeld, Germany}
\email{hofmanova@math.uni-bielefeld.de}
\thanks{M.H. gratefully acknowledges the financial support by the German Science Foundation DFG via the Collaborative Research Center SFB1283.}

\begin{abstract}

To circumvent the ill-posedness issues present in various models of continuum fluid mechanics, we present a dynamical systems approach  aiming at selection of  physically relevant solutions. Even under the presence of infinitely many solutions  to the full Euler system describing the motion of a compressible inviscid fluid, our approach permits to select a system of solutions (one trajectory for every initial condition) satisfying the classical semiflow property. 
Moreover, the selection  respects  the well accepted admissibility criteria for physical solutions, namely,  maximization of the entropy production rate and the weak--strong uniqueness principle. Consequently, strong solutions are always selected whenever they exist and stationary states are stable and included in the selection as well. To this end, we introduce a notion of \emph{dissipative solution}, which is given by a triple of density, momentum and total entropy defined as expectations of a suitable measure--valued solution.


%

\end{abstract}


\date{\today}

\maketitle

\tableofcontents

\section{Introduction}
\label{I}

The \emph{Euler system} describing the motion of a general compressible inviscid fluid represents one of the basic 
models in the framework of continuum fluid mechanics. The unknown fields are the fluid density 
$\vr = \vr(t,x)$, the momentum $\vm = \vm(t,x)$, and the energy $\mathcal{E} = \E(t,x)$
satisfying the system of partial differential equations: 
\begin{equation} \label{I1}
\begin{split}
\partial_t \vr + \Div \vm &= 0,\\
\partial_t \vm + \Div \left( \frac{\vm \otimes \vm}{\vr} \right) + \Grad p &= 0,\\
\partial_t \E + \Div \left[ \left( \E + p \right) \frac{\vm}{\vr} \right] &= 0.
\end{split}
\end{equation}
Writing the energy as a sum of its kinetic and internal components, 
\[
\E = \frac{1}{2} \frac{ |\vm|^2}{\vr} + \vr e,
\]
we suppose that the pressure $p$ and the internal energy $e$ satisfy 
the caloric equation of state in the form
\begin{equation} \label{I2}
(\gamma - 1) \vr e = p,\ \mbox{where}\ \gamma > 1 \ \mbox{is the adiabatic constant.}
\end{equation}
In addition, we introduce the absolute temperature $\vt$ through the Boyle--Mariotte thermal equation of state: 
\begin{equation} \label{I3}
p = \vr \vt\ \mbox{yielding} \ e = c_v \vt ,\ c_v = \frac{1}{\gamma - 1}.
\end{equation}
Supposing that the fluid occupies a bounded spatial domain $\Omega \subset R^N$, $N=1,2,3$ we impose the impermeability 
boundary condition
\begin{equation} \label{I4}
\vm \cdot \vc{n}|_{\partial \Omega} = 0.
\end{equation}
Finally, the initial state of the fluid is given through the initial conditions 
\begin{equation} \label{I5}
\vr(0, \cdot) = \vr_0,\ \vm(0, \cdot) = \vm_0, \ \mathcal{E}(0, \cdot) = \mathcal{E}_0.
\end{equation}
The Second law of thermodynamics is enforced through the entropy balance equation 
\begin{equation} \label{I6}
\partial_t (\vr s) + \Div (s \vm ) = 0 \ \mbox{or}\ \partial_t s + \left( \frac{ \vm }{\vr} \right) \cdot \Grad s = 0,
\end{equation}
where the entropy $s$ is given as
\[
s(\vr, \vt) = \log(\vt^{c_v}) - \log(\vr).
\]
There is a vast amount of literature dedicated to the mathematical theory of the Euler system. In particular, it is known that 
the initial--value problem is well posed \emph{locally} in time in the class of smooth solutions, see e.g.
the monograph by Majda \cite{Majd} or the more recent treatment by 
Benzoni--Gavage and Serre \cite{BenSer}. Smooth solutions, however, develop singularities in a finite time for a fairly 
general class of initial data, see e.g. Smoller \cite{SMO}. Thus if the Euler system is accepted as an adequate description of the 
fluid motion in a long run, a concept of generalized solution is needed. 

The modern theory of partial differential equations is based on the concept of weak solution, where the derivatives are understood 
in the sense of distributions. This gives rise to a large class of objects in which \emph{uniqueness} might be lost. Several admissibility criteria have been proposed to select the physically relevant weak solution, among which the 
\emph{entropy inequality} 
\begin{equation} \label{I7}
\partial_t (\vr s) + \Div (s \vm ) \geq 0 
\end{equation}
reflecting the Second law of thermodynamics. The recent adaptation of the method of convex integration, 
developed in the context of incompressible fluids by De Lellis and Sz\' ekelyhidi \cite{DelSze3}, gave rise to numerous examples of ill--posedness also in the class of compressible fluids, see Chiodaroli, De Lellis, Kreml
\cite{ChiDelKre}, Chiodaroli and Kreml \cite{ChiKre} , Chiodaroli et al. \cite{ChKrMaSwI}, among others. In particular, it was 
shown in \cite{FeKlKrMa} that the Euler system \eqref{I1}--\eqref{I5} is ill--posed, specifically it admits infinitely many 
weak solutions for a large class of initial data. Moreover, these solutions satisfy the entropy inequality 
\eqref{I7}. In addition, examples of regular initial data producing infinitely many weak solutions in the long run have been 
also obtained  in \cite{FeKlKrMa}.

\subsection{Admissibility criteria}
\label{AD}

In view of these facts, more refined admissibility criteria are needed in order to select the physically relevant solutions. Dafermos 
\cite{Dafer} proposed a selection criterion based on \emph{maximality} of the total entropy. We formulate it in a slightly stronger form, avoiding the issue of existence of right-sided time derivatives. Specifically, suppose that 
$[ \vr_i, \vm_i, \E_i ]$, $i = 1,2$ are two solutions of the Euler system \eqref{I1}--\eqref{I5} with the associated total 
entropies 
\[
S_i = \vr_i s(\vr_i, \vm_i, \E_i),\ i = 1,2.
\]
We say that
\[
[ \vr_1, \vm_1, \E_1 ] {\succ}_{\mathcal{D}} [ \vr_2, \vm_2, \E_2 ]
\] 
if:
\begin{itemize}
\item 
there exists $\tau \geq 0$ such that 
\[
[ \vr_1(t, \cdot), \vm_1(t, \cdot), \E_1(t,\cdot) ]
= [ \vr_2(t, \cdot), \vm_2(t, \cdot), \E_2(t,\cdot) ] \ \mbox{for any}\ t \in [0,\tau];
\]
\item
there exists $\delta > 0$ such that 
\begin{equation}\label{eq:D}
\intO{ S_1(t+, x) } \geq \intO{ S_2(t+, x) } \ \mbox{for all}\ t \in (\tau, \tau + \delta). 
\end{equation}
\end{itemize}
A weak solution $[ \vr, \vm, E]$ is called maximal (admissible) if it is maximal with respect to the relation 
${\succ}_{\mathcal{D}}$.

A  modification of this criterion was further investigated in the context of the barotropic Euler system in \cite{EF2013_N2}. However, in this case it is rather the total energy that shall be minimized, according to the principle of maximal energy dissipation. The criterium of \cite{EF2013_N2}  translated to our setting of the full Euler system, leads to a weaker version of the condition \eqref{eq:D}. Namely, we require that there exists a sequence $(\tau_{n})_{n}$, $\tau_{n}> \tau$, $\tau_{n}\to \tau$ such that 
\begin{equation}\label{eq:D1}
\intO{ S_1(\tau_{n}+, x) } \geq \intO{ S_2(\tau_{n}+, x) }\ \mbox{for all}\ n\in\mathbb{N}.
\end{equation}
It was shown in \cite{EF2013_N2} that the  solutions constructed there by the method of convex integration   do not fulfill the corresponding criterium for maximal energy dissipation, suggesting that such a criterium shall be retained in order to exclude  nonphysical solutions.

To compare these two criteria, let us denote by  ${\succ}_{\mathcal{F}}$ the partial ordering induced by \eqref{eq:D1}. The following holds true: If a solution is maximal with respect to ${\succ}_{\mathcal{F}}$ then it is also maximal with respect to ${\succ}_{\mathcal{D}}$. This may seem surprising at first sight since \eqref{eq:D} is obviously a stronger condition than \eqref{eq:D1}. But it is exactly for this reason why the implication of maximality is valid. More precisely, the weaker condition \eqref{eq:D1} allows to compare more solutions, for instance also those  that oscillate around each other and that are therefore not comparable by the condition \eqref{eq:D}.

An alternative criterion enforcing the Second law of thermodynamics is \emph{maximility of the global entropy production} proposed in   
\cite{BreFei17B}. In accordance with the Schwartz representation principle, the inequality in the 
entropy balance \eqref{I7} can be interpreted as 
\[
\partial_t S + \Div \left( S \frac{\vm}{\vr} \right) = \rmk{\Sigma}, \ S = \vr s,
\]  
where $\Sigma$ is a non--negative Borel measure on $[0, \infty) \times \Ov{\Omega}$. Similarly to the above, we say that 
\begin{equation} \label{I8}
[ \vr_1, \vm_1, \E_1 ]{\succ} [ \vr_2, \vm_2, \E_2 ]
\ \Leftrightarrow \Sigma_1 \geq \Sigma_2 \ \mbox{on any compact subset of}\ [0, \infty) \times \Ov{\Omega}.
\end{equation}
It was shown that maximal solutions with respect to the relation $\succ$ exist for the full Euler system under rather general hypotheses on the data as well as  the constitutive relations.

Despite these efforts, however, none of the above selection criteria proved  sufficient to guarantee the desired well-posedness result. The major open question therefore remains: how can physically relevant solutions be distinguished from the nonphysical ones? The aim of the present paper is to take into account another physical property of an evolution system, namely, the so-called semiflow property: starting the system at time $0$, letting it run to time $s>0$ and then restarting and continue for the time $t>0$, the state of the system at the final time $s+t$ should be the same as if  the system ran directly from $0$ to  $s+t$. If uniqueness holds, the semiflow property immediately follows. However, for systems where uniqueness is unknown or not valid, it is generally not clear whether such a semiflow even exists.



\subsection{Semiflow solutions}\label{SS1}

On the following pages we show how to approach this problem and we construct a solution semiflow to the complete Euler system \eqref{I1}--\eqref{I4}. In particular, this leads to a well-defined dynamical system associated to \eqref{I1}--\eqref{I4}, which depends in a measurable way on the initial condition. To this end, it is in the first place necessary to identify the correct phase variables of the system together with a suitable notion of solution. Even though the system \eqref{I1} describes the evolution of the density, momentum and energy, it turns out to be beneficial to replace the energy by the total entropy.
In other words, the state of the fluid at a given time 
$t \geq 0$ will be determined by the values of three phase variables, 
\[
\mbox{the density}\ \vr(t, \cdot), \ \mbox{the momentum}\ \vm(t,\cdot),\ \mbox{the total entropy}\ S (t,\cdot) = \vr s (t, \cdot),
\]
interpreted through their spatial integral means as quantities in suitable abstract function spaces.

The reason why we prefer the 
entropy $S$ instead of the energy $\E$ is the lack of suitable {\it a~priori} bounds for the latter. The integral means 
\[
t \mapsto \intO{ \vr(t, \cdot) \varphi },\ t \mapsto \intO{ \vm(t, \cdot) \cdot \bfphi }
\] 
will be continuous for $t \in [0, \infty)$ for any smooth $\varphi$, $\bfphi$, while
\[
t \mapsto \intO{ S (t, \cdot) \varphi } \in \BV_{{\rm loc}}[0, \infty); 
\] 
whence the limits 
\[
\intO{ S (t+, \cdot) \varphi },\ \intO{ S (t-, \cdot) \varphi } 
\ \mbox{are well defined for}\ t \in [0,\infty)
\]
with the convention 
\[
S(0-) = S_0.
\]
Hence, another difficulty regarding the classical theory of dynamical systems  stems from the fact that solutions are not continuous in time, which  in particular holds here for the  total entropy. We overcome this issue by replacing continuity by existence of  one--sided limits at every time $t\geq0$.

As the next step, we shall determine what input information on the initial state of the system is necessary. Apart from the initial state for the density, momentum and total entropy, we shall be also given the total energy $E_{0}$ which is a constant of the motion and provides various bounds for all the corresponding quantitities. Given the initial state of the system $\mathbb{U}_0=(\vc{U}_0,E_0)$, where 
\begin{equation} \label{I9}
\vc{U}_0 = [\vr_0, \vm_0, S_0], \ \mbox{together with the total energy}\ 
E_0 \geq \intO{ \left[ \frac{|\vm_0|^2}{\vr_0} + \vr_0 e(\vr_0, S_0) \right] }, 
\end{equation}
we then identify in a unique way the state of the system at the time $t \geq 0$, 
\begin{align*}
&\mathbb{U}[t; (\vc{U}_0, E_0) ] = \Big([\vr(t, \cdot), \vm(t, \cdot), S(t-, \cdot)], E_0 \Big), 
\\ &E_0 \geq \intO{ \left[ \frac{|\vm(t, \cdot)|^2}{\vr(t, \cdot)} + \vr(t, \cdot) e\Big(\vr(t, \cdot), S (t \pm, \cdot) \Big) \right] }.
\end{align*}
In addition, 
the mapping $t \mapsto \mathbb{U}[t; (\vc{U}_0, E_0) ]$ will enjoy the \emph{semiflow property}:
\begin{itemize}
\item 
\[
\mathbb{U}[0; (\vc{U}_0, E_0) ] = (\vc{U}_0, E_0) 
\]
\item
\[
\mathbb{U}\Big[t_1 + t_2; (\vc{U}_0, E_0) \Big] = \mathbb{U}\Big[ t_1; \mathbb{U}[t_2, (\vc{U}_0, E_0) ] \Big] 
\ \mbox{for any}\ t_1, \ t_2 \geq 0.
\]

\end{itemize} 

\noindent
In particular, the state of the system at a time $t \geq 0$ is uniquely determined in terms of the initial data 
$[\vr_0, \vm_0, S_0]$ and the initial energy $E_0$.

The trajectories
\[
t \mapsto \mathbb{U}\Big[t; (\vc{U}_0, E_0) \Big] = 
\Big([ \vr(t, \cdot), \vm(t, \cdot), S(t-, \cdot)], E_0 \Big)
\]
represent a generalized - dissipative - solution to the Euler system \eqref{I1}--\eqref{I5}, which we introduce in Section \ref{W}. It will be shown that they comply with the 
following stipulations:

\begin{itemize}

\item \emph{Weak--strong uniqueness.} Suppose that the Euler system \eqref{I1}--\eqref{I5} 
admits a classical solution $[\vr, \vm, \E]$ on a time interval $[0, T_{\rm max})$, 
\[
E_0 = \intO{ \E_0 }.
\]

Then for $\vc{U}_0 = [\vr_0, \vm_0, S(\vr_0, \vm_0, \E_0)]$ we have
\[
\mathbb{U} \Big[ t; (\vc{U}_0, E_0) \Big] = 
\left( \Big[\vr(t, \cdot), \vm(t, \cdot), S(\vr(t,\cdot), \vm(t, \cdot), \E(t, \cdot) ) \Big], E_0 \right)
\]
for all $t \in [0, T_{\rm max})$.

\item \emph{Maximal  entropy production.}
Suppose that 
\[
\left( [\vr_1, \vm_1, S_1], E_0 \right) = \mathbb{U} \Big[\cdot\,; (  \vc{U}_0 , E_0 ) \Big]  
\]
and that $\left( [\vr_2, \vm_2, S_2], E_0 \right)$ is another dissipative solution starting from the same initial data 
$(  \vc{U}_0, E_0 )$ and such that
\[
\sigma_1(t):=\int_{\Omega}(S_{1}(t)-S_{0})\,\dif x \leq \sigma_2(t) :=\int_{\Omega}(S_{2}(t)-S_{0})\,\dif x \ \text{on}\ [0,\infty),
\]
where $\sigma_i$ is the entropy production rate associated to $[\vr_i, \vm_i, S_i]$, $i = 1,2$, respectively. 

Then 
\[
\sigma_1 = \sigma_2.
\] 

\item \emph{Stability of stationary states.} 
Let $\vr = \Ov{\vr} > 0$, $\vm = 0$, $ \Ov{\E}  = \frac{1}{|\Omega|} E_0$ be a stationary solution of the Euler system \eqref{I1}. Suppose that the system reaches the equilibrium state at some time $T\geq0$, i.e.
\[
\mathbb{U} \left[ T; \Big( [\vr_0, \vm_0, S_0]; E_0 \Big) \right] = 
\left( \Big[ \Ov{\vr}, 0, S(\Ov{\vr}, 0, \Ov{\E}) \Big], E_0 \right) .
\]

Then
\[
\mathbb{U} \left[ t; \Big( [\vr_0, \vm_0, S_0]; E_0 \Big) \right] = 
\left( \Big[ \Ov{\vr}, 0, S(\Ov{\vr}, 0, \Ov{\E}) \Big], E_0 \right) 
\]
for all $t \geq T$.

\end{itemize}

\subsection{Selection procedure}

The semiflow $\mathbb{U}$ will be constructed by means of the selection procedure originally proposed in the context of stochastic Markov 
processes by Krylov \cite{KrylNV}. The method was later adapted to deterministic evolutionary problems with time continuous solutions by Cardona and Kapitanski \cite{CorKap}. We have developed a similar approach to the \emph{isentropic} Euler system in \cite{BreFeiHof19}, where we also relaxed the continuity assumption on the trajectories. We refer the reader to  Section~\ref{dD} for details of the abstract formulation.

In Section \ref{W}, 
we specify a class of generalized solutions to the Euler system \eqref{I1}--\eqref{I5} termed \emph{dissipative solutions}. They are, 
loosely speaking, the expected values of suitable measure--valued solutions as introduced in \cite{BreFei17} but with an additional refinement regarding the associated concentration defect measures.

Applying carefully the selection procedure of \cite{KrylNV}, \cite{CorKap}, \cite{BreFeiHof19}, we obtain the desired semiflow solution in Section~\ref{S}.
The paper is concluded by a  discussion of further extensions in Section~\ref{D}.  In particular, in Section \ref{sec:D} we compare our construction to the two versions of Dafermos' criterium introduced in Section \ref{AD}.
 
\section{Dynamical systems approach}
\label{dD}

We start by adapting the general approach of \cite{KrylNV}, \cite{CorKap} to problems with discontinuous solutions paths, see also \cite{BreFeiHof19}.
Suppose that the state of a physical system at each instant $t \geq 0$ is characterized by an abstract vector $\mathbb{U}(t) \in \mathcal{S}$ ranging in a suitable \emph{phase space} $\mathcal{S}$. In our setting, the state space will be a separable Hilbert space or a Polish topological space. 

Consider a mapping  
\[
\mathbb{U}: [0, \infty) \times \mathcal{S} \to \mathcal{S},\ [t, \mathbb{U}_0] \mapsto \mathbb{U}[t; \mathbb{U}_0],
\]
meaning $\mathbb{U}[t; \mathbb{U}_0]$ is the state of the system emanating 
from the initial state $\mathbb{U}_0$ 
at a time $t \geq 0$.

\begin{Definition}[Semiflow] \label{dDD1}

We say that a mapping $\mathbb{U}: [0, \infty) \times \mathcal{S} \to \mathcal{S}$ is a \emph{semiflow} if:
\begin{enumerate}
\item[(a)] The flow starts at the initial datum, that is
\[
\mathbb{U}[0, \mathbb{U}_0] = \mathbb{U}_0 \ \mbox{for any} \ \mathbb{U}_0 \in \mathcal{S}.
\]

\item[(b)] It has the semigroup property, that is
\[
\mathbb{U}[t_1 + t_2, \mathbb{U}_0] = \mathbb{U}\left[ t_2 , \mathbb{U}[t_1, \mathbb{U}_0] \right] \ \mbox{for any}\ \mathbb{U}_0 \in \mathcal{S},\ t_1,t_2\geq0.
\]

\end{enumerate}

\end{Definition}

In the classical theory of dynamical systems, \emph{continuity} of the semiflow $\mathbb{U}$ is required both in $t$ and in $\mathbb{U}_0$. In view of the issues discussed in the previous section, we develop a generalized theory, where continuity in time is relaxed to the existence of the one--sided limits at $t \pm$, 
while continuity in $\mathbb{U}_0$ is replaced by \emph{measurability} with respect to Borel sets generated by suitable topologies.

\subsection{Trajectory space}

The trajectories space can be defined as
$
\mathcal{T} =  \BV_{\rm loc}([0, \infty); \mathcal{S}).
$

For trajectories we can define the following operations:
\begin{itemize}
\item {\bf Time shift.} For $\xi \in \mathcal{T}$ and $T> 0$ we set
\[
\mathfrak S_T [\xi] (t \pm ) = \xi ((T + t)\pm) \ \mbox{for}\ t \in [0, \infty) 
\]
\item {\bf Continuation.} For $\xi^1, \ \xi^2 \in \mathcal{T}$ and $T>0$ we set
\[
\ \xi^1 \cup_T \xi^2 (t \pm) = \left\{ \begin{array}{l} \xi^1(t \pm ) \ \mbox{if}\ 0 < t < T, \\ \\
\xi^2 ((t - T) \pm ) \ \mbox{for}\ t > T,\end{array}   \right.,\ 
\begin{array}{l}
\xi^1 \cup_T \xi^2 (T -) = \xi^1(T-),\\ \\ \xi^1 \cup_T \xi^2 (T+) = \xi^2(0+).
\end{array}
\]

\end{itemize}

\subsection{Solution space}

The solution space can be loosely described as the family of all solutions of a given system of equations emanating from a fixed initial datum in $\mathcal{S}$.
As such the solution operator $\mathcal{U}$ can be understood as a mapping
\[
\mathcal{U} : \mathcal{S} \to 2^{\mathcal{T}},\ \mathcal{U}[ \mathbb{U}_0 ] \subset \mathcal T, 
\ \mathbb{U} [ \mathbb{U}_0 ] (0-) = \mathbb{U}_0 \ \mbox{for all}\ \mathbb{U} \in \mathcal{U}[ \mathbb{U}_0 ].
\]

\begin{Definition}[Solution operator] \label{dDD2}

We say that a mapping $\mathcal{U}: \mathcal{S} \to 2^{\mathcal{T}}$ is \emph{solution operator} if the following properties are satisfied:
\begin{enumerate}
\item[{\bf [A1]}] {\bf Existence, compactness.}
For each $\mathbb{U}_0 \in \mathcal{S}$ the set
$\mathcal U[\mathbb{U}_0]$ is a non--empty bounded subset of $ \BV_{\rm loc}([0, \infty); \mathcal{S})$, compact in the topology of $L^1_{{\rm loc}}([0, \infty); \mathcal{S})$,
\[
\mathbb{U} [\mathbb{U}_0](0-) = \mathbb{U}_0 \ \mbox{for any}\ \mathbb{U} \in \mathcal{U}[\mathbb{U}_0].
\]
\item[{\bf [A2]}] {\bf Measurability.} The set valued mapping
\[
\mathcal{U}: \mathbb{U}_0 \in \mathcal{S} \mapsto \mathcal{U}[ \mathbb{U}_0] \in 2^\mathcal{T}
\]
is Borel measurable; where $2^\mathcal{T}$ is endowed with the Hausdorff topology defined on compact subsets of the metric space
$L^1_{{\rm loc}}([0, \infty); \mathcal{S})$.

\item[{\bf [A3]}] {\bf Shift invariance.} For any $\mathbb{U} \in \mathcal{U} [ \mathbb{U}_0]$ and any $T > 0$, we have
\[
\mathfrak S_T [\mathbb{U} ] \in \mathcal{U} \left[ \mathbb{U} (T-) \right].
\]
\item[{\bf [A4]}] {\bf Continuation.} For any $\mathbb{U} ^1 \in \mathcal{U}[\mathbb{U}_0]$, $T > 0$, and $\mathbb{U}^2 \in \mathcal{U}[\mathbb{U} ^1(T-)]$, we have
\[
\mathbb{U}^1 \cup_T \mathbb{U}^2 \in \mathcal{U}[\mathbb{U}_0].
\]

\end{enumerate}

\end{Definition}

A version of the following result for the case of $\mathcal{T}$ being a set of continuous trajectories with values in a Polish space was proved by Cardona and Kapitanski \cite{CorKap}. However,  for applications in fluid dynamics, where time continuity of the energy or entropy is not valid, a suitable modification is necessary which was proved in \cite[Section 5]{BreFeiHof19}.

\begin{Proposition} \label{dDP1}

Let $\mathcal{S}$ be a Polish space and let $\mathcal{T} =  \BV_{\rm loc}([0, \infty); \mathcal{S})$.
Let $\mathcal{U}: \mathcal{S} \to 2^\mathcal{T}$, be a solution operator enjoying the properties {\bf [A1]--[A4]} from Definition \ref{dDD2}.

Then $\mathcal{U}$ admits a measurable semiflow selection, specifically, for any $\mathbb{U}_0 \in \mathcal{S}$ there exists a single trajectory
\[
\mathbb{U} \in \mathcal{U}[\mathbb{U}_0],\ \mathbb{U}: \mathcal{S}  \to \mathcal{T}
\subset L^1_{\rm loc}([0, \infty); \mathcal{S})\ \mbox{Borel measurable},
\]
such that 
\[
\mathbb{U}: [t,\mathbb{U}_0] \in [0, \infty) \times  \mathcal{S} \mapsto \mathbb{U}[\mathbb{U}_0](t-),\ t> 0,\  \mathbb{U}(0-) = \mathbb{U}_0
\]
is a semiflow in the sense of Definition \ref{dDD1}.

In addition, if $\beta : \mathcal{S} \to R$ is a bounded continuous function and $\lambda > 0$, the selection can be chosen to satisfy
\begin{equation}\label{eq:1}
\int_0^\infty \exp(- \lambda t) \beta (\mathbb{U}[t; \mathbb{U}_0]) \dt
\leq \int_0^\infty \exp(- \lambda t) \beta (\mathbb{V}(t)) \dt
\ \mbox{for any}\ \mathbb{V} \in \mathcal{U}[\mathbb{U}_0].
\end{equation}

\end{Proposition}

We recall that the selection procedure of Proposition \ref{dDP1}  relies on a subsequent minimization (or alternatively maximization) of a sequence of suitable continuous functionals, as e.g. the functional in \eqref{eq:1}, over the set of all solutions. The functionals are chosen in a way to separate points of the trajectory space $\mathcal{T}$. For this purpose the Laplace transform with respect to the time variable proved to be beneficial. In general, the selection depends on the particular choice of a sequence $(\lambda_{n})_{n\in\mathbb{N}}$ representing the points where the Laplace transform is evaluated. The functional in \eqref{eq:1} shall then be chosen as the first one to minimize/maximize. In our application to the complete Euler system in the next section we chose \eqref{eq:1} in order to maximize the entropy production rate.

In the remaining part of the paper, we identify a suitable solution operator $\mathcal{U}$ associated to the Euler system that 
complies with the hypotheses of Proposition \ref{dDP1}.  In Section \ref{D}, we present further discussion on the question of optimality of the choice of the functionals and in particular the relation to the Dafermos criterion discussed in the Introduction.

\section{Dissipative solutions}
\label{W}

Our first goal in this section is to identify a suitable class of generalized solutions to the Euler problem \eqref{I1}--\eqref{I5}.
We start by introducing the basic state variables, then we recall the (slightly modified) notion of dissipative measure--valued solutions from \cite{BreFei17}. This leads us to our notion of dissipative solution which we introduce in Section \ref{DDD}. Afterwards we study stability of dissipative solutions as well as their existence. The semiflow is then constructed in Section \ref{S}.

\subsection{Phase space}

In what follows, we plan to work with the state variables $\vr$, $\vm$, and the total entropy $S = \vr s$. In accordance with 
hypotheses \eqref{I2}, \eqref{I3}, the
pressure $p$ and the internal energy $e$ can be written in the form
\[
p = p(\vr, S) = \vr^\gamma \exp\left( \frac{S}{c_v \vr} \right),\ e=e(\vr,S) = \frac{1}{\gamma - 1} \vr^{\gamma-1} \exp\left( \frac{S}{c_v \vr} \right).
\]

\begin{Lemma} \label{LC1}
The mapping
\[
(\vr, S) \mapsto p(\vr, S),\ \vr > 0, \ S \in R,
\]
is strictly convex.

\end{Lemma}

\begin{proof}

This is a matter of direct computation of the Hessian matrix. We have
\[
\begin{split}
\frac{\partial p(\vr, S)}{\partial \vr} &= \gamma \vr^{\gamma - 1} \exp\left( \frac{S}{c_v \vr} \right)
- \frac{S}{c_v }\vr^{\gamma - 2} \exp\left( \frac{S}{c_v \vr} \right),\\
\frac{\partial p(\vr, S)}{\partial S} &= \frac{1}{c_v} \vr^{\gamma - 1} \exp\left( \frac{S}{c_v \vr} \right),
\end{split}
\]
and
\[
\begin{split}
\frac{\partial^2 p(\vr, S)}{\partial \vr^2} &= \left[ (\gamma-1)\vr^2 + \left( (\gamma -1) \vr -
\frac{S}{c_v} \right)^2 \right] \vr^{\gamma - 4} \exp\left( \frac{S}{c_v \vr} \right), \\
\frac{\partial^2 p(\vr, S)}{\partial S^2} &= \frac{1}{c^2_v} \vr^{\gamma - 2} \exp\left( \frac{S}{c_v \vr} \right)
= \frac{1}{c^2_v} \vr^{2} \vr^{\gamma - 4} \exp\left( \frac{S}{c_v \vr} \right)
,\\
\frac{\partial^2 p(\vr, S)}{\partial \vr \partial S} &=
\frac{1}{c_v} (\gamma - 1) \vr^{\gamma - 2} \exp\left( \frac{S}{c_v \vr} \right)  - \frac{S}{c_v^2} \vr^{\gamma - 3 } \exp\left( \frac{S}{c_v \vr} \right) \\&= \left[\frac{1}{c_v} (\gamma - 1) \vr^{2}  - \frac{S}{c_v^2} \vr               \right] \vr^{\gamma - 4} \exp\left( \frac{S}{c_v \vr} \right).
\end{split}
\]

Obviously the Hessian matrix has positive trace, while its determinant reads
\[
\begin{split}
&\vr^{\gamma - 4} \exp\left( \frac{S}{c_v \vr} \right) \left\{
\frac{\vr^2}{c_v^2} \left[ (\gamma-1)\vr^2 + \left( (\gamma -1) \vr -
\frac{S}{c_v} \right)^2 \right] - \left( \frac{\gamma - 1}{c_v}  \vr^{2}  - \frac{S}{c_v^2} \vr  \right)^2 \right\}\\
&\qquad= \frac{\vr^{\gamma}}{c_v^2} \exp\left( \frac{S}{c_v \vr} \right).
\end{split}
\]
which completes the proof.
\end{proof}

Consequently, we may define 
\[
p(\vr, S) = (\gamma - 1) \vr e(\vr, S) = \left\{ \begin{array}{l} \vr^\gamma \exp \left( \frac{S}{c_v \vr} \right) 
\ \mbox{if}\ \vr > 0,\ S \in R ,\\ \\ 0 \ \mbox{if}\ \vr = 0, \ S \leq 0, \\ \\ 
\infty \ \mbox{if} \ \vr =0,\ S > 0,  \end{array}  \right. 
\]
which is a convex lower semi--continuous function on $[0, \infty) \times R$. 

Similarly, we define the kinetic energy 
\[
\frac{|\vm|^2}{\vr} = \left\{ \begin{array}{l} \frac{|\vm|^2}{\vr} \ \mbox{if}\ \vr > 0,\ \vm \in R^N,\\ \\
0 \ \mbox{if}\ \vm = 0,\\ \\ \infty \ \mbox{if}\ \vr = 0, \vm \ne 0,
\end{array} \right.
\]
which is  a convex lower semi--continuous function on $[0, \infty) \times R^N$. We conclude that the total energy 
\[
\E = \E(\vr, \vm, S) = \frac{1}{2} \frac{|\vm|^2}{\vr} + \vr e(\vr, S)
\]
may be viewed as a convex lower semi--continuous function of $\vr \geq 0$, $\vm \in R^N$, and $S \in R$. 

\subsection{Dissipative measure--valued solutions}

Following \cite{BreFei17} we introduce the concept of dissipative measure--valued solution to the Euler problem 
\eqref{I1}--\eqref{I5}. In addition, similarly to \cite{BreFeiHof19}, we refine the definition of the measures 
describing the concentration defect. The reader may consult \cite{BreFei17} for the physical background and 
mathematical objects like Young measures used in what follows.  

We start by introducing the state space of ``dummy variables'': 
\[
\mathcal{Q}= \left\{ [\tilde \vr, \tilde \vm, \tilde S] \ \Big|\ \tilde \vr \geq 0, \ \tilde \vm \in R^N, \
\tilde S \in R \right\}.
\]
By $\mathcal{P}(\mathcal{Q})$ we denote the set of probability measures on $\mathcal{Q}$, whereas $\mathcal{M}^{+}(\Ov{\Omega})$ and $\mathcal{M}^{+}(S^{N-1}\times\Ov{\Omega})$ denotes the set of positive bounded Radon measures on $\Ov{\Omega}$ and $S^{N-1}\times\Ov{\Omega}$, respectively, where $S^{N-1}\subset R^{N}$ is the unit sphere.

A \emph{dissipative measure--valued solution} of the Euler system \eqref{I1}--\eqref{I4} with the initial data 
\[
[\vr_0, \vm_0, S_0] \ \mbox{and the total energy} \ E_0
\]
consists of the following objects:
\begin{itemize}
\item a parametrized family of probability measures 
\[
\mathcal{V}_{t,x}: (t,x) \in (0, \infty) \times \Omega \mapsto \mathcal{P}(\mathcal{Q}),
\ \mathcal{V} \in L^\infty_{{\rm weak-(*)}} ((0,T) \times \Omega; \mathcal{P}(\mathcal{Q}));
\]
\item kinetic and internal energy concentration defect measures 
\[
\mathfrak{C}_{{\rm kin}},\ \mathfrak{C}_{{\rm int}} \in L^\infty_{{\rm weak-(*)}}(0, \infty; \mathcal{M}^+ (\Ov{\Omega}));
\]
\item convective concentration defect measure 
\[
\mathfrak{C}_{\rm conv} \in L^\infty_{\rm weak- (*)}(0, \infty; \mathcal{M}^+ (S^{N-1} \times \Ov{\Omega})),\
\frac{1}{2}\int_{S^{N-1}} \D \mathfrak{C}_{\rm conv} = \mathfrak{C}_{\rm kin}.
\]

\end{itemize}

The Euler equations are satisfied in the following sense:

\begin{itemize}

\item \emph{Total energy} is a constant of motion: 
\begin{equation} \label{D2}
\intO{ \left< \mathcal{V}_{t,x}; \frac{1}{2} \frac{|\tilde \vm|^2}{\tilde \vr} + c_v {\tilde \vr}^\gamma \exp \left( \frac{\tilde S}{c_v \tilde \vr} \right) \right>}
+ \int_{\Ov{\Omega}} (\D \mathfrak{C}_{\rm kin}(t) + \D \mathfrak{C}_{\rm int}(t)) = E_0
\ \mbox{for a.a.}\ t \geq 0.
\end{equation}

\item \emph{Mass conservation} or continuity equation $\mbox{\eqref{I1}}_1$ reads 
\begin{equation} \label{D3}
\int_0^\infty \intO{ \left[ \left< \mathcal{V}_{t,x}; \tilde {\vr} \right>
\partial_t \varphi + \left< \mathcal{V}_{t,x}; \tilde {\vm} \right> \cdot \Grad \varphi \right] }\dt
= - \intO{ \vr_0 \varphi (0)}
\end{equation}
for any $\varphi \in C^1_c([0, \infty) \times \Ov{\Omega})$.

\item \emph{Momentum balance} (equation $\mbox{\eqref{I1}}_2$) reads
\begin{equation} \label{D4}
\begin{split}
&\int_0^\infty \intO{ \left[ \left< \mathcal{V}_{t,x}; \tilde {\vm} \right> \cdot
\partial_t \bfphi + \left< \mathcal{V}_{t,x}; \frac{\tilde {\vm} \otimes \tilde{\vm} }{\tilde{\vr}} \right> : \Grad \bfphi
+ \left< \mathcal{V}_{t,x}; p(\tilde{\vr}, \tilde{S}) \right> \Div \bfphi
\right] }\dt\\
&+ \int_0^\infty \left[ \int_{\Ov{\Omega}} \int_{S^{N-1}} (\xi \otimes \xi): \Grad \bfphi \ \D \mathfrak{C}_{\rm conv}(t) \right] \dt  +
(\gamma - 1) \int_0^\infty \left[ \int_{\Ov{\Omega}} \Div \bfphi \ \D \mathfrak{C}_{\rm int}(t) \right] \dt \\ 
&= - \intO{ \vm_0 \bfphi (0) } 
\end{split}
\end{equation}
for any $\bfphi \in C^1_c([0, \infty)\times\overline\Omega;R^N)$, $\bfphi \cdot \vc{n}|_{\partial \Omega} = 0$.

\item \emph{Entropy balance} (inequality \eqref{I7}) is rewritten in the renormalized form:
\begin{equation}\label{D5}
\begin{split}
\int_0^\infty &\intO{ \left[ \left< \mathcal{V}_{t,x}; \tilde \vr Z \left( \frac{\tilde S}{\tilde \vr} \right) \right> \partial_t \varphi +
\left< \mathcal{V}_{t,x}; Z \left( \frac{\tilde S}{\tilde \vr} \right) \tilde \vm \right> \cdot \Grad \varphi  \right]}\dt
\\
&\leq - \intO{ \vr_0 Z \left( \frac{S_0}{\vr_0} \right) \varphi (0) } 
\end{split}
\end{equation}
for any $\varphi \in C^1_c([0, \infty) \times \Ov{\Omega})$, $\varphi \geq 0$, and any $Z$,
\[
Z \in BC(R)\ \mbox{non--decreasing}.
\]
\end{itemize}

It follows from \eqref{D3}, \eqref{D4} that 
\[
\begin{split}
t &\mapsto \intO{ \left< \mathcal{V}_{t,x}; \tilde \vr \right> \varphi}, \ \varphi \in C^1(\Ov{\Omega}), \\ 
t &\mapsto \intO{ \left< \mathcal{V}_{t,x}; \tilde \vm \right> \bfphi}, \ \bfphi \in C^{1}(\Ov{\Omega};R^N),
\ \bfphi \cdot \vc{n}|_{\partial \Omega} = 0,
\end{split}
\]
are continuous functions of time. Accordingly, equations \eqref{D3}, \eqref{D4} can be written as 
\begin{align} \label{D3a}
\int_{\tau_1}^{\tau_2} \intO{ \left[ \left< \mathcal{V}_{t,x}; \tilde {\vr} \right>
\partial_t \varphi + \left< \mathcal{V}_{t,x}; \tilde {\vm} \right> \cdot \Grad \varphi \right] }\dt
= \left[ \intO{ \left< \mathcal{V}_{t,x}; \tilde {\vr} \right> \varphi } \right]_{t = \tau_1}^{t=\tau_2},
\end{align}
for any $0 \leq \tau_1 < \tau_2 < \infty$, and
any $\varphi \in C^1_c([0, \infty) \times \Ov{\Omega})$, where $\left< \mathcal{V}_{0,x}; \tilde {\vr} \right> = \vr_0(x)$;
\begin{equation} \label{D4a}
\begin{split}
&\int_{\tau_1}^{\tau_2} \intO{ \left[ \left< \mathcal{V}_{t,x}; \tilde {\vm} \right> \cdot
\partial_t \bfphi + \left< \mathcal{V}_{t,x}; \frac{\tilde {\vm} \otimes \tilde{\vm} }{\tilde{\vr}} \right> : \Grad \bfphi
+ \left< \mathcal{V}_{t,x}; p(\tilde{\vr}, \tilde{S}) \right> \Div \bfphi
\right] }\dt\\
&+ \int_{\tau_1}^{\tau_2} \left[ \int_{\Ov{\Omega}} \int_{S^{N-1}} (\xi \otimes \xi): \Grad \bfphi \ \D \mathfrak{C}_{\rm conv}(t) 
\right] \dt +
(\gamma - 1) \int_{\tau_1}^{\tau_2} \left[ \int_{\Ov{\Omega}} \Div \bfphi \ \D \mathfrak{C}_{\rm int}(t) \right] \dt \\ 
&= \left[ \intO{ \left< \mathcal{V}_{t,x} ;
\tilde \vm \right> \bfphi }  \right]_{t=\tau_1}^{t= \tau_2}, \ \left< \mathcal{V}_{0,x} ;
\tilde \vm \right> = \vm_0(x)
\end{split}
\end{equation}
for any $0 \leq \tau_1 < \tau_2 < \infty$,
and any $\bfphi \in C^1_c([0, \infty)\times\Ov\Omega; R^N)$, $\bfphi \cdot \vc{n}|_{\partial \Omega} = 0$.

As shown in \cite[Section 2.1.1]{BreFei17}, the renormalized entropy inequality \eqref{D5} implies that 
\begin{equation} \label{D5b}
\mathcal{V}_{t,x} \left\{ (\vr, S) \ \Big| \ {S} \geq s_0 \vr > - \infty \right\} = 1 
\ \mbox{for a.a.} \ (t,x)
\end{equation}
as soon as 
\[
{S_0} \geq \vr_0 s_0 \ \mbox{a.a. in}\ \Omega.
\]
This is the minimum principle for the entropy $s = \frac{S}{\vr} \geq s_0$. From now on, we fix $s_0$ and consider only solutions 
satisfying \eqref{D5b}. This corresponds to having a new entropy $s - s_0 \geq 0$. Then one can perform the limit passage 
$Z(s) \nearrow s$ in \eqref{D5} obtaining the entropy balance 
\begin{equation}\label{D5a}
\begin{split}
\int_0^\infty &\intO{ \left[ \left< \mathcal{V}_{t,x}; \tilde S \right> \partial_t \varphi +
\left< \mathcal{V}_{t,x}; \tilde  S \frac{\tilde \vm}{\tilde \vr} \right> \cdot \Grad \varphi  \right]}\dt
\leq - \intO{ S_0 \varphi (0) } 
\end{split}
\end{equation}
cf. \cite{BreFei17B}. In particular, 
\[
t \mapsto \intO{ \left< \mathcal{V}_{t,x}; \tilde S \right> \varphi } = \chi^1_\varphi (t) + \chi^2_\varphi (t) ,\ 
\varphi \in C^1 (\Ov{\Omega}), \ \varphi \geq 0, 
\]
where $\chi^1$ is continuous and $\chi^2$ non--decreasing. Thus \eqref{D5a} can be rewritten in the form 
\begin{equation}\label{D5c}
\begin{split}
\int_{\tau_1}^{\tau_2} \intO{ \left[ \left< \mathcal{V}_{t,x}; {\tilde S} \right> \partial_t \varphi +
\left< \mathcal{V}_{t,x}; \tilde{S} \frac{\tilde \vm }{\tilde \vr} \right> \cdot \Grad \varphi  \right]}
\leq \left[ \intO{ \left< \mathcal{V}_{t,x}; {\tilde S} \right> \varphi  } \right]_{t = \tau_1-}^{t = \tau_2+}, 
\end{split}
\end{equation}
for any $0 \leq \tau_1 < \tau_2 < \infty$,
and any $\varphi \in C^1([0, \infty) \times \Ov{\Omega})$, $\varphi \geq 0$, where $\left< \mathcal{V}_{0-,x}; {\tilde S} \right> = S_0(x)$.

Now we have all in hand to formulate the definition of dissipative measure--valued solution.

\begin{Definition}[Dissipative measure--valued solution] \label{DD1}

A \emph{dissipative measure--valued solution} of the Euler system \eqref{I1}--\eqref{I4}
with the initial data 
\[
[\vr_0, \vm_0, S_0] \ \mbox{and the energy} \ E_0
\]
is a parameterized family of probability measures
\[ 
\mathcal{V}_{t,x}: (t,x) \in (0, \infty) \times \Omega \mapsto \mathcal{P}(\mathcal{Q}),
\ \mathcal{V} \in L^\infty_{{\rm weak-(*)}} ((0,T) \times \Omega; \mathcal{P}(\mathcal{Q})),
\]
together with
the energy concentration defect measures 
\[
\mathfrak{C}_{{\rm kin}},\ \mathfrak{C}_{{\rm int}} \in L^\infty_{{\rm weak-(*)}}(0, \infty; \mathcal{M}^+ (\Ov{\Omega})),
\]
and the convection concentration defect measure 
\[
\mathfrak{C}_{\rm conv} \in L^\infty_{\rm weak- (*)}(0, \infty; \mathcal{M}^+ (S^{N-1} \times \Ov{\Omega})),\
\frac{1}{2}\int_{S^{N-1}} \D \mathfrak{C}_{\rm conv} = \mathfrak{C}_{\rm kin},
\]
satisfying the integral identities \eqref{D2}, \eqref{D3a}, \eqref{D4a}, and \eqref{D5c}. 

\end{Definition}

Next, we list certain bounds that can be derived 
from hypothesis \eqref{D5b} and the energy equality \eqref{D2}.
As the entropy is bounded below by $s_0$, 
we deduce from \eqref{D2} that
\begin{equation} \label{D6}
\intO{ \left< \mathcal{V}_{t,x}; {\tilde \vr}^\gamma \right> } \aleq E_0 \ \mbox{for a.a.}\ t > 0.
\end{equation}

Similarly, as
\[
|\tilde{\vm}|^{\frac{2 \gamma}{\gamma + 1}} = |\tilde{\vr}|^{\frac{\gamma}{\gamma + 1}} \left| \frac{\tilde{\vm}}{\sqrt{\tilde \vr}} \right|^{\frac{2 \gamma}{\gamma + 1}}
\aleq \tvr^\gamma + \frac{ |\tilde \vm|^2}{\tvr},
\]
we conclude
\begin{equation} \label{D7}
\intO{ \left< \mathcal{V}_{t,x}; |\tilde{\vm}|^{\frac{2 \gamma}{\gamma + 1}}  \right> } \aleq E_0 \ \mbox{for a.a.}\ t > 0.
\end{equation}

Finally, we derive some bounds on the total entropy $S$. Recall that $\tilde S \geq s_0 \tvr$; whence
\[
|\tilde S| = - \tilde S \leq - s_0 \tvr\ \mbox{whenever}\ \tilde S \leq 0.
\]
If $\tilde S > 0$, we compute
\[
\tvr^\gamma \exp \left( \frac{\tilde S}{c_v \tvr} \right) = c_v^{-\gamma} \frac { \exp \left( \frac{\tilde S}{c_v \tvr} \right) }
{ \left( \frac{\tilde S}{c_v \tvr} \right)^\gamma } {\tilde S}^\gamma \ageq {\tilde S}^\gamma.
\]
Consequently,
\begin{equation} \label{D8}
\intO{ \left< \mathcal{V}_{t,x}; |\tilde S|^\gamma  \right> } \aleq E_0 \ \mbox{for a.a.}\ t > 0.
\end{equation}
Next, we estimate the quantity $\tilde S/ \sqrt{\tvr}$. If $\tilde S \leq 0$, we get, repeating the above argument,
\[
\left| \frac{ \tilde S }{\sqrt{\tvr}} \right| \leq - s_0 \sqrt{\tvr}
\ \mbox{for} \ \tilde S \leq 0.
\]
If $\tilde S > 0$, we write
\[
\tvr^\gamma \exp \left( \frac{\tilde S}{c_v \tvr} \right) =
\tvr^\gamma \exp \left( {\frac{\tilde S}{\sqrt{\tvr}}} \frac{1}{c_v \sqrt {\tvr} } \right) =
c_v^{-2\gamma} \frac{\exp \left( {\frac{\tilde S}{\sqrt{\tvr}}} \frac{1}{c_v \sqrt {\tvr} } \right)   }{ \left( { {\frac{\tilde S}{\sqrt{\tvr}}} }\frac{1} { c_v {\sqrt {\tvr}} }
\right)^{2 \gamma} } \left(\frac{\tilde S}{\sqrt{\tvr}}\right)^{2 \gamma} \ageq \left(\frac{\tilde S}{\sqrt{\tvr}}\right)^{2 \gamma}.
\]
We may therefore infer that
\begin{equation} \label{D9}
\intO{ \left< \mathcal{V}_{t,x}; \left| \frac{\tilde S}{\sqrt{\tvr}} \right|^{2\gamma}  \right> } \aleq E_{0} \ \mbox{for a.a.}\ t > 0.
\end{equation}

Finally, we recall the \emph{weak--strong} uniqueness principle proved in \cite[Theorem 3.3]{BreFei17}.

\begin{Proposition} \label{DPP2}

Let $\Omega \subset R^N$, $N=1,2,3$ be a bounded domain with smooth boundary. Suppose that the Euler system \eqref{I1}--\eqref{I5}
admits a classical solution $[\vr, \vm, \E]$ in the class 
\begin{equation} \label{class}
\vr,\ \E \in C([0,T); W^{3,2}(\Omega)),\ \vm \in C([0,T); W^{3,2}(\Omega; R^N))
\end{equation} 
with the initial data 
\[
\vr_0 > 0,\ \vm_0, \ \E_0 = \frac{1}{2} \frac{|\vm_0|^2}{\vr_0} + c_v \vr_0 \vt_0, \ \vt_0 > 0.
\]

Let $\mathcal{V}_{t,x}$ be a dissipative measure valued solution specified in Definition \ref{DD1} starting from the data 
\[
\vr_0,\ \vm_0, \ S_0 = \vr_0 s(\vr_0, \vt_0),\ E_0 = \intO{ \E_0 }.
\]   

Then 
\[
\mathfrak{C}_{\rm kin}|_{[0,T) \times \Ov{\Omega}} =
\mathfrak{C}_{\rm int}|_{[0,T) \times \Ov{\Omega}}=0,\   \mathfrak{C}_{\rm conv}|_{ [0,T) \times S^{N-1} \times \Ov{\Omega}} = 0,
\]
and
\[
\mathcal{V}_{t,x} = \delta_{[\vr(t,x), \vm(t,x), S(t,x)]}
\]
for a.a. $(t,x) \in [0,T) \times \Omega$. 

\end{Proposition}

Note that \emph{existence} of a local--in--time classical solution in the class \eqref{class} was established by Schochet \cite{SCHO1}.

\subsection{Dissipative solution}\label{DDD}

Having collected the necessary preliminary material, we are ready to introduce the central object of the present paper - 
the \emph{dissipative} solutions to the Euler system.

\begin{Definition}[Dissipative solution] \label{DDS1}

The quantity $([\vr, \vm, S], E_0)$, where
\begin{align} \label{weak}
\begin{aligned}
&\vr \in C_{{\rm weak, loc}} ([0, \infty); L^\gamma(\Omega)),\
\vm \in C_{{\rm weak, loc}} ([0, \infty); L^{\frac{2 \gamma}{\gamma + 1}}(\Omega; R^N)),\\
& S \in L^\infty(0,\infty;L^\gamma(\Omega))\cap\BV_{{\rm weak, loc}} ([0, \infty); W^{-\ell,2}(\Omega)),\ \ell>\frac{N}{2}+1,
\end{aligned}
\end{align}
is a \emph{dissipative solution} of the Euler system \eqref{I1}
with the initial data
\[
([\vr_0, \ \vm_0,\ S_0], E_0)\in L^\gamma(\Omega)\times L^{\frac{2\gamma}{\gamma+1}}(\Omega,R^N)\times L^\gamma(\Omega)\times[0,\infty)
\]
if there exists a dissipative measure--valued solution $\mathcal{V}_{t,x}$ as specified in Definition \ref{DD1} such that
\[
\vr(t,x) = \left< \mathcal{V}_{t,x}; \tvr \right>,\
\vm (t,x) = \left< \mathcal{V}_{t,x}; \tilde \vm \right>, \
S(t,x) = \left< \mathcal{V}_{t,x}; \tilde S \right>.
\]

\end{Definition}

\begin{Remark} \label{RRR2}

In accordance with the bounds \eqref{D2}, \eqref{D6}, \eqref{D7}, and \eqref{D8}, any dissipative solution belongs to 
the class 
\[
\vr \in L^\infty ([0, \infty); L^\gamma(\Omega)),\
\vm \in L^\infty ([0, \infty); L^{\frac{2 \gamma}{\gamma + 1}}(\Omega; R^N)),\
S \in L^\infty ([0, \infty); L^\gamma(\Omega)); 
\] 
whence to bounded balls in the afore mentioned spaces. These are compact metric (Polish) spaces with respect to the weak topology.  
In particular, condition \eqref{weak} reduces to 
\begin{equation} \label{weak1}
\begin{split}
t &\mapsto \intO{ \vr \varphi },\ t \mapsto \intO{ \vm \cdot \bfphi} \in C_{\rm loc}[0, \infty) 
\ \mbox{for any}\ \varphi \in \DC(\Omega),\ \bfphi \in \DC(\Omega; R^N),\\
t &\mapsto \intO{ S \varphi } \in \BV_{{\rm loc}}[0, \infty) \ \mbox{for any}\ \varphi \in \DC(\Omega).
\end{split}
\end{equation}

\end{Remark}

Finally, we introduce a subclass of dissipative solutions that reflect the physical principle of maximal dissipation defined via the entropy production rate
\begin{align*}
\sigma(\tau)=\int_\Omega \big(S(\tau)-S_0\big)\dx
\end{align*}
discussed in Section \ref{SS1}. 
Let $([\vr^i, \vm^i,S^i], E_{0})$, $i=1,2$, be two dissipative solutions starting from the same initial data $([\vr_0, \vm_0,S_0] ,E_0)$ with entropy production rates $\sigma^1$ and $\sigma^2$. Similarly to \eqref{I8}, we introduce the relation
\begin{equation*} 
([\vr^1, \vm^1,S^1], E_{0}) \succ ([\vr^2, \vm^2,S^2], E_{0})\ \Leftrightarrow \ \sigma^{1}(\tau \pm) \geq \sigma^{2}(\tau \pm)
\ \mbox{for any}\ \tau \in (0, \infty).
\end{equation*}

\begin{Definition}[Maximal dissipative solution] \label{DD2}
A dissipative solution $([\vr, \vm,S], E_0)$ starting from the initial data $([\vr_0, \vm_0, S_0],E_0)$ is a \emph{maximal dissipative solution} if
it is maximal with respect to the relation $\succ$. Specifically, if
\[
([\tvr, \tvm,\tilde S], {E}_{0}) \succ ([\vr, \vm, S],E_{0}),
\]
where $([\tvr, \tvm, \tilde S],{E}_{0})$ is another dissipative solution starting from $([\vr_0, \vm_0,S_0], E_0)$, then
\[
\sigma = \tilde \sigma  \ \mbox{in}\ [0, \infty).
\]
Here $\sigma$ and $\tilde\sigma$ are the entropy production rates of $([\tvr, \tvm, \tilde S],{E}_{0})$ and $([\vr, \vm, S],E)$ respectively.
\end{Definition}

\subsection{Sequential stability}

We start by introducing suitable topologies on the space of the initial data and the space of dissipative solutions.
Fix  $\ell>N/2+1$ and consider the Hilbert space
\[
\mathcal S_{\rm Euler} = W^{-\ell,2}(\Q) \times W^{-\ell,2}(\Q; R^N)\times W^{-\ell,2}(\Q)  \times R,
\]
together with its subset containing the initial data
\begin{equation*}
\begin{split}
\mathcal D_{\rm Euler} = &\left\{ ([\vr_0, \vm_0, S_0],E_0) \in L^1(\Omega; R^{N+2}) 
\times R \ \Big| \right. \\ &\left. \vr_0 \geq 0,\, S_0\geq s_0 \varrho_0,\,
\intQ{ \left[\frac{1}{2} \frac{| \vm_0|^2}{ \vr_0} + c_v {\vr}_0^\gamma \exp \left( \frac{
 S_0}{c_v  \vr_0} \right)\right]} \leq E_0 \right\} 
\end{split}
\end{equation*}
Note that  $\mathcal D_{\rm Euler}$ is a closed convex subset of $\mathcal S_{\rm Euler}$. We also define the trajectory space
$$\mathcal{T}_{\rm Euler} =  \BV_{\rm loc}([0, \infty); \mathcal{S}_{\rm Euler}).$$ 

In the following we are going to show sequential stability (compactness) of the set of dissipative solutions. This will be subsequently used in the proof of existence of dissipative solutions (using a suitable approximation) as well as measurability of the mapping
\[
([\vr_0, \vm_0,S_0], E_0) \in \mathcal{D}_{\rm Euler} \mapsto \mathcal{U}([\vr_0, \vm_0,S_0], E_0) \in 2^{\mathcal T_{\rm Euler}},\
\mathcal T_{\rm Euler}=L^1_{{\rm loc}}([0, \infty); \mathcal{S}_{\rm Euler}),
\]
where $\mathcal{U}( [\vr_0, \vm_0,S_0], E_0)$ denotes the solution set
\begin{equation*} 
\begin{split}
\mathcal{U} &([\vr_0, \vm_0,S_0], E_0) = 
\left\{ [\vr, \vm,S, E_0] \in \mathcal{T}_{\rm Euler} \ \Big| \right.\\ &
([\vr, \vm,S], E_0) \ \mbox{is a dissipative solution with initial data}\ ([\vr_0, \vm_0,S_0], E_0) \Big\}
\end{split}
\end{equation*}
for the initial data $([\vr_0, \vm_0,S_0], E_0) \in \mathcal D_{\rm Euler}$.
 We have the following result.
\begin{Proposition} \label{ESP2}

Suppose that $\{([ \vr_{0,\ep}, \vm_{0,\ep},S_{0,\ep}], E_{0,\ep} )\}_{\ep > 0} \subset \mathcal D_{\rm Euler}$ is a sequence of data giving rise to a family of dissipative solutions
$\{ ([\vre, \vme,S_\ep], E_{0,\ep}) \}_{\ep > 0}$, that is,
$$
([\vre, \vme,S_\ep], E_{0,\ep}) \in \mathcal{U}([ \vr_{0,\ep}, \vm_{0,\ep}, S_{0,\ep}], E_{0,\ep} ).
$$
Moreover, we assume that there  exists $\Ov{E}>0$ such that $  E_{0,\ep} \leq \Ov{E}$ for all $\ep > 0$.

Then, at least for suitable subsequences,
\begin{align} \label{ES1}
\begin{aligned}
&\vr_{0,\ep} \to \vr_0 \ \mbox{weakly in}\ L^\gamma(\Omega), \
\vm_{0,\ep} \to \vm_0 \ \mbox{weakly in}\ L^{\frac{2 \gamma}{\gamma + 1}}(\Omega; R^N),\\
&\qquad\qquad S_{0,\ep} \to S_0 \ \mbox{weakly in}\ L^\gamma(\Omega),\ E_{0,\ep} \to E_0.
\end{aligned}
\end{align}
and
\[
\begin{split}
\vre &\to \vr \ \mbox{in}\ C_{{\rm weak, loc}}([0, \infty); L^\gamma (\Omega)),\\
\vme &\to \vm \ \mbox{in}\ C_{{\rm weak, loc}}([0, \infty); L^{\frac{2 \gamma}{\gamma + 1}} (\Omega; R^N)),\\
S_\ep &\rightarrow S \ \mbox{in}\ L^q_{\rm loc} ([0,\infty);W^{-\ell,2}(\Omega))\ \mbox{for any}\ q<\infty,\\
S_\ep &\rightarrow S \ 
\mbox{weakly-(*) in}\ L^\infty(0, \infty; L^{\gamma}(\Omega)),
\end{split}
\]
where
\[
([\vr, \vm,S], E_0) \in \mathcal{U} ([\vr_0, \vm_0,S_0], E_0).
\]

\end{Proposition}
\begin{proof}
We mainly follow the ideas from \cite[Proposition 3.1]{BreFeiHof19},  which we refer to for further details. Some modifications are needed to accommodate the additional variable $S$ (in the internal energy and the entropy balance).

First, we claim that
the convergence in \eqref{ES1} follows immediately from the boundedness of the initial energy $E_{0,\ep}$ and the uniform bounds 
\eqref{D6}--\eqref{D8}. Indeed, 
using Jensen's inequality and convexity of the energy established in Lemma \ref{LC1}, we deduce that
\begin{equation}\label{eq:J1}
\left[ \frac{1}{2} \frac{|\vme|^2}{\vre} + p(\vre,S_\ep) \right] (t,x)
\leq \left< \mathcal{V}_{t,x}^\ep ; \frac{1}{2} \frac{|\tvm|^2}{\tvr} + p( \tvr,\tilde S)\right> \ \mbox{for a.a.}\ (t,x) \in 
(0, \infty) \times \Q,
\end{equation}
where $\mathcal{V}^{\ep}$ is the (Young) measure associated with the solution $([\vr_{\ep},\vm_{\ep},S_\ep],E_{0,\ep})$.
Integrating over $\Omega$ we can see that the right hand side is bounded by  $E_{0,\ep} \leq \Ov{E}$ using \eqref{D2}.

On the other, we can use \eqref{D6}--\eqref{D9} to obtain the estimates
\begin{align}\label{eq:2002a}
\sup_{t>0}\int_\Omega\left[\frac{|\vme|^{2}}{\vre}+\vre^\gamma+S^\gamma_\ep \right]\dx&\leq \,c(\Ov{E}),\\
 \sup_{t>0}\int_\Omega\left[|\vme|^{2\gamma/\gamma+1}+\Big|\frac{S_\ep}{\sqrt{\vre}}\Big|^{2\gamma} \right]\dx&\leq \,c(\Ov{E}),\label{eq:2002b}
\end{align}
uniformly in $\ep$.
Consequently, we deduce from equations \eqref{D3}, \eqref{D4}, and the energy balance \eqref{D2}, that
\[
\begin{split}
\vre &\to \vr \ \mbox{in}\ C_{{\rm weak, loc}}([0, \infty); L^\gamma(\Omega)), \ \vr \geq 0,\\
\vme &\to \vm \ \mbox{in}\ C_{{\rm weak, loc}}([0, \infty); L^\frac{2\gamma}{\gamma+1}(\Omega; R^N)),
\end{split}
\]
where
\[
\vr(0, \cdot) = \vr_0, \ \vm(0, \cdot) = \vm_0.
\]

Similarly, we get from \eqref{eq:2002a} that 
\[
S_\ep \to S \ \mbox{weakly-(*) in}\ L^\infty(0, \infty; L^{\gamma}(\Omega)).
\]
Moreover, we deduce from the entropy balance \eqref{D5c} and Helly's selection theorem that 
\[
\intO{ S_{\ep}(t, \cdot) \varphi } \to \intO{ S(t, \cdot) \varphi } 
\ \mbox{for any}\ t > 0 \ \mbox{and any}\ \varphi \in W^{\ell,2}(\Omega),
\]
modifying $S$ on the set of times of measure zero as the case may be (we split $\varphi=\varphi^+-\varphi^-$ and argue for both separately). Note that it is enough to show the former one for a dense subset of $W^{\ell,2}(\Omega)$ which follows from a diagonal argument. Seeing that 
$L^{\gamma}(\Omega)$ endowed with the weak topology is compactly embedded in $W^{-\ell,2}(\Omega)$ we obtain the desired conclusion 
\[
S_\ep \to S \ \mbox{in}\ L^q_{\rm loc}(0, \infty; W^{-\ell,2}(\Omega)) \ \mbox{for any}\ 1 \leq q < \infty.
\]

Obviously, 
\[
E_{0,\ep} \to E_0 
\]
passing to a subsequence as the case may be. Next, 
it is easy to perform the limit in the equation of continuity \eqref{D3} to obtain 
\eqref{D3a}. 

Next, as a consequence of the energy balance \eqref{D2}, the Young measures $\mathcal{V}^{\varepsilon}_{t,x}$ have uniformly bounded first moments; whence
\begin{equation}\label{eq:2}
\mathcal{V}^\ep_{t,x}  \to \mathcal{V}_{t,x}
\ \mbox{weakly-(*) in}\ L^\infty_{\rm weak-(*)} \left((0, \infty ) \times \Q; \mathcal{P}(\mathcal{S}) \right)
\end{equation}
at least for a subsequence. In addition, using estimates \eqref{D7}--\eqref{D9} we can pass to the limit in the entropy balance 
\eqref{D5c}. 

The limit in the total energy balance \eqref{D2} and the momentum balance \eqref{D4} 
involves the concentration measures and is more technical. However, the procedure is exactly the same as in \cite[Section 3]{BreFeiHof19}
therefore we omit the proof here. 
\end{proof}

\subsection{Existence of dissipative solutions}
The sequential stability from the previous part combined with a suitable approximation implies the existence of a dissipative solution. The precise statement is the content of the following assertion.
\begin{Proposition} \label{ESP1}

Let $([\vr_0,\vc{m}_0,S_0],E_0)\in \mathcal D_{\rm Euler}$ be given.
Then the Euler system \eqref{I1}--\eqref{I4} admits a dissipative solution in the sense of Definition
\ref{DDS1} with the initial data
$([\vr_0, \vm_0,S_0], E_0)$.

\end{Proposition}

\begin{proof} 
Given initial data $([\vr_0,\vc{m}_0,S_0],E_0) \in \mathcal D_{\rm Euler}$ implies
\[
\vr_0 \in L^\gamma(\Q), \ \vm_0 \in L^{\frac{2\gamma}{\gamma + 1}}(\Q; R^N), \ S_0\in L^\gamma(\Q),\
\vr_0 \geq 0,\ S_0 \geq s_0 \vr_0.
\]
with the respective bounds in terms of $E_{0}$, recall the lower bounds for the energy obtained in \eqref{D6}--\eqref{D8}.
It is standard to construct smooth functions $\vr_{0,\ep}$, $\vm_{0,\ep}$ and $S_{0,\ep}$, where $\vr_{0,\ep}$ is strictly positive and $\vm_{0,\ep}$ compactly supported in $\Omega$, $S_{0,\ep} \geq \vr_{0,\ep} s_0$ such that
\[
\vr_{0,\ep} \to \vr_0 \ \mbox{in}\ L^\gamma(\Q),\
\vm_{0,\ep} \to \vm_0 \ \mbox{in}\ L^{\frac{2\gamma}{\gamma + 1}}(\Q; R^N),\ S_{0,\ep} \to S_0 \ \mbox{in}\ L^\gamma(\Q),
\]
and
\[
\intQ{ \left[ \frac{1}{2}\frac{|\vm_{0,\ep}|^2}{\vr_{0,\ep}} +  c_v \vr^\gamma_{0,\ep} \exp\left( \frac{S_{0,\ep}}{c_v \vr_{0,\ep}} \right)\right] } \to\intQ{ \left[ \frac{1}{2}\frac{|\vm_{0}|^2}{\vr_{0}} + c_v \vr^\gamma_{0} \exp\left( \frac{S_{0}}{c_v \vr_{0}} \right)\right] } 
\]
as $\ep \to 0$. Finally, we set $s_{0,\ep}=S_{0,\ep}/\vr_{0,\ep}$ and $\vu_{0,\ep}=\vm_{0,\ep}/\vr_{0,\ep}$.

Now, similarly to Kr\"oner and Zajaczkowski \cite{KrZa}, we consider the approximate system
\begin{equation} \label{ES11}
\begin{split}
\partial_t \vr + \Div (\vr \vu) &= 0,\\
\partial_t (\vr \vu) + \Div (\vr \vu \otimes \vu) + \Grad\left(\vr^\gamma \exp\left( \frac{s}{c_v} \right)\right) &= \varepsilon
\mathcal{L} \vu,\ \ep > 0,\\
\partial_t s+\vu\cdot\nabla s&=0,
\end{split}
\end{equation}
with the initial data 
\begin{equation} \label{ES12}
\vr(0, \cdot) = \vr_{0,\ep},\ \vu(0, \cdot) = \vu_{0,\ep}, \ s(0,\cdot)=s_{0,\ep}.
\end{equation}

Here, the symbol $\mathcal{L}$ denotes a suitable ``viscosity'' operator, with the associated set of boundary conditions to be imposed on the velocity field $\vu$. In \cite{KrZa}, the authors consider $\mathcal{L} = \Del^3$ with the \emph{Dirichlet} boundary conditions. 
This is not convenient here, as the resulting admissible test functions $\bfphi$ in the momentum equation \eqref{D4} would be compactly
supported in $\Omega$. In order to preserve the weak--strong uniqueness principle (Proposition \ref{DPP2}), however, we need a larger 
class of test functions satisfying merely the impermeability condition $\bfphi \cdot \vc{n}|_{\partial \Omega} = 0$. To this end, we use a different ansatz inspired by Kato and Lai \cite{KaLai}.

Let $W^{3,2}_n(\Omega; R^N)$ be the Hilbert space, 
\[
W^{3,2}_n (\Omega; R^N) = \left\{ \vu \in W^{3,2}(\Omega; R^N) \ \Big|\ \vu \cdot \vc{n}|_{\partial \Omega} = 0 \right\}.
\]
We suppose that $\Omega \subset R^N$, $N=1,2,3$ is a bounded regular domain so that 
$W^{3,2}_n(\Omega; R^N)$ is compactly embedded in $C^1(\Ov{\Omega}; R^N)$. Moreover, $W^{3,2}_n$ being a closed subset of the 
separable Hilbert space $W^{3,2}$ is separable. Let $(( \ ; \ ))$ be the scalar product on $W^{3,2}_n$, that is,
\[
(( \vc{v}; \vc{w} )) = \sum_{|\alpha| = 3} \intO{ \partial^\alpha_x \vc{v} \cdot \partial^\alpha_x \vc{w} } 
+ \intO{ \vc{v} \cdot \vc{w} },\ \vc{v},  \vc{w} \in W^{3,2}_n(\Omega; R^N).
\]   
Similarly to \cite{KaLai}, we consider $\mathcal{L}$ to be the self--adjoint operator on $W^{3,2}_n$ associated to the bilinear form 
$(( \ ; \ ))$, namely 
\[
\mathcal{L} \vu = \sum_{|\alpha| = 3} (-\partial^\alpha_x) \partial^\alpha_x \vu + \vu 
\ \mbox{with the homogeneous Neumann boundary conditions.}
\]
The associated variational formulation of the momentum equation in \eqref{ES11} then reads 
\begin{equation} \label{ES11b}
\left[ \intO{ \vr \vu \cdot \bfphi } \right]_{t = 0}^{t = \tau} - 
\intO{ \left[ \vr \vu \otimes \vu : \Grad \bfphi + \vr^\gamma \exp \left( \frac{s}{c_v} \right) \Div \bfphi \right] } 
= - \ep (( \vu; \bfphi ))
\end{equation}
for any $\tau > 0$, and any $\bfphi \in W^{3,2}_n(\Omega; R^N)$.

As $W^{3,2}_n$ is a separable Hilbert space, the existence proof used in \cite[Theorem 4.1]{KrZa} applies 
without essential modifications (see also Hoff \cite{Hoff10} for how to handle the transport equations for $\vr$ and $s$). 
Given $\ep > 0$, there exists a trio $(\vr_\ep,\vu_\ep,s_\ep)$ of the following class
which is a classical solution to \eqref{ES11} in the following sense:
\begin{itemize}
\item The balance of mass \eqref{ES11}$_1$ and the entropy balance \eqref{ES11}$_3$ hold pointwise in $(0,\infty)\times\Omega$.
\item The balance of momentum \eqref{ES11}$_2$ holds as \eqref{ES11b}.
\item The solution possesses  enough regularity to rigorously justify the standard energy estimate
\begin{equation} \label{ES13}
\frac{{\rm d}}{\dt} \intQ{ \left[ \frac{1}{2}\vr_{\ep} |\vu_{\ep}|^2 + c_v \vr^\gamma_{\ep} \exp\left( \frac{s_{\ep}}{c_v} \right)\right] }+\ep ((\vu_\ep; \vu_\ep ))^2 = 0.
\end{equation}

\end{itemize}

Finally, we set 
$S_\ep=\vr_\ep s_\ep$ and $\vm_\ep=\vre\vue$. 
Using the arguments of the preceding section, it is easy to perform the limit $\ep \to 0$ in the sequence of approximate solutions
\[
\left\{ \vre, \vm_\ep = \vre \vue,S_\ep=\vr_\ep s_\ep, E_\ep = \intQ{ \left[ \frac{1}{2}\vr_{\ep} |\vu_{\ep}|^2+c_v \vr^\gamma_{\ep} \exp\left( \frac{S_{\ep}}{c_v \vr_{\ep}} \right) \right] }\right\}_{\ep > 0}
\]
to obtain the desired dissipative solution, with the exception of the energy equality \eqref{D2} that now reads 
\begin{equation} \label{D2a}
\intO{ \left< \mathcal{V}_{t,x}; \frac{1}{2} \frac{|\tilde \vm|^2}{\tilde \vr} + c_v {\tilde \vr}^\gamma \exp \left( \frac{\tilde S}{c_v \tilde \vr} \right) \right>}
+ \int_{\Ov{\Omega}} (\D \mathfrak{C}_{\rm kin}(t) + \D \mathfrak{C}_{\rm int}(t)) \leq E_0
\ \mbox{for a.a.}\ t \geq 0.
\end{equation}
Note that, on account of \eqref{ES13}, it is easy to see that viscous term in the momentum equation vanishes as $\ep\to 0$.
In contrast with \eqref{D2a}, the entropy balance \eqref{D5c} holds as \emph{equality}. To convert \eqref{D2a} to equality, 
it is enough to augment $\mathfrak{C}_{\rm int}(t)$ by $h(t) \dx$ with a suitable spatially homogeneous $h \geq 0$. Note that the momentum equation 
\eqref{D4} remains valid as $\mathfrak{C}_{\rm int}(t)$ acts on $\Div \bfphi$, where $\bfphi \cdot \vc{n}|_{\partial \Omega} = 0$; 
whence
\[
\intO{ h(t) \Div \bfphi } = 0.
\] 
\end{proof}

\section{Semiflow selection}
\label{S}

The goal of this section is to show that there is a semiflow selection to the Euler system \eqref{I1}. We recall that the precise definition of a semiflow - in the abstract framework - is given in Definition \ref{dDD1}.  
The following is the main result of the present paper.

\begin{Theorem}[Semiflow solution] \label{ST1}
The Euler system \eqref{I1}--\eqref{I4} admits a semiflow solution
in the class of dissipative solutions in the sense of Definition \ref{DDS1}. More specifically, 
for any initial data 
\[
\mathbb{U}_0 = \left( [ \vr_0, \vm_0, S_0 ], E_0 \right) \in \mathcal{D}_{\rm Euler} 
\]
there exists a dissipative solution 
\[
\mathbb{U} = \mathbb{U}[t, \mathbb{U}_0] = 
\left( \big[ \vr(t, \cdot), \vm(t, \cdot), S(t-, \cdot) \big], E_0 \right) \in 
\mathcal{T}_{\rm Euler}
\]
enjoying the following properties:
\begin{itemize} 
\item for each $\mathbb{U}_0 \in \mathcal{D}_{\rm Euler}$ the solution 
$\mathbb{U}[\cdot , \mathbb{U}_0]$ is maximal in the sense of Definition \ref{DD2};
\item $\mathcal{D}_{\rm Euler}$ is an invariant set, meaning 
\[
\mathbb{U}[t, \mathbb{U}_0] \in \mathcal{D}_{\rm Euler} 
\]
for any $t \geq 0$;
\item the mapping 
\[
\mathbb{U}_0 \in \mathcal{S}_{\rm Euler} \mapsto \mathbb{U}[\cdot, \mathbb{U}_0] 
\in L^1_{\rm loc}(0, \infty; \mathcal{S}_{\rm Euler})
\]
is Borel measurable; 
\item the mapping
\[
\mathbb{U}: [0, \infty) \times \mathcal{D}_{\rm Euler} \mapsto \mathcal{D}_{\rm Euler}
\]
is a semiflow, specifically,
\[
\mathbb{U}[t_1 + t_2; \mathbb{U}_0] = \mathbb{U}\left[ t_2; \Big[ \mathbb{U}[t_1; \mathbb{U}_0 \Big] \right] 
\]
for any $t_1, t_2 \geq 0$, and any $\mathbb{U}_0 \in \mathcal{D}_{\rm Euler}$.
\end{itemize}

\end{Theorem}
The claim of Theorem \ref{ST1} will follow from Proposition \ref{dDP1} as soon as we verify {\bf [A1]}--{\bf [A4]}
from Definition \ref{dDD2} for the solution set $\mathcal U([\vr_0, \vm_0,S_0], E_0)$ with $([\vr_0, \vm_0,S_0], E_0) \in \mathcal D_{\rm Euler}$. This will be done in the following lemmas.

\begin{Lemma}\label{LDS0}
For each $([\vr_0, \vm_0,S_0], E_0)\in \mathcal{D}_{\rm Euler}$ the set
$\mathcal U[\vr_0, \vm_0,S_0, E_0]$ is a non--empty bounded subset of $\BV_{\rm loc}([0, \infty); \mathcal{S}_{\rm Euler})$ and it is compact in the topology of $L^1_{{\rm loc}}([0, \infty); \mathcal{S}_{\rm Euler})$.
\end{Lemma}
\begin{proof}
The claim follows from Propositions \ref{ESP2} and \ref{ESP1}.
\end{proof}

\begin{Lemma} \label{LDS1}

Let 
$([\vr, \vc{m}, S], E_{0})$ be a dissipative solution to the Euler system in the sense of Definition \ref{DDS1}. 
Then $\mathfrak{S}_T \circ ([\vr, \vc{m}, S], E_{0})$ is a dissipative solution corresponding to the data $([\vr, \vc{m}, S](T-) ,E_{0}) $.
 
\end{Lemma}
\begin{proof}

We recall that the operation of \emph{time shift} is given by
\[
\mathfrak{S}_T \circ ([\vr, \vc{m}, S], E_0) (t) = ([\vr, \vc{m}, S ](T + t), E_0), \ T > 0, \ t \geq 0.
\]
Since a shift of a test function in \eqref{D2}--\eqref{D5} produces a test function in the same class, the claim easily follows.
\end{proof}

\begin{Lemma} \label{LDS2}

Let $([\vr^1, \vc{m}^1, S^1], E_{0})$ be a dissipative solution of the Euler system, $T > 0$ and let $([\vr^2, \vc{m}^2, S^2], E_{0})$ be another 
dissipative solution with the initial data $([\vr^1, \vc{m}^1, S^1] (T-), E_{0})$. 

Then 
\[
([\vr^1, \vc{m}^1, S^1], E_{0}) \cup_T ([\vr^2, \vc{m}^2, S^2], E_{0}) 
\]
is a dissipative solution of the Euler system.

\end{Lemma}

\begin{proof}
The concatenation property for \eqref{D2} and \eqref{D3} is obvious.
As far as \eqref{D4} and \eqref{D5} are concerned we observe the following.
The energy is convex lower semicontinuous, in as much we obtain
\[
\begin{split}
&\intO{ \left[ \frac{1}{2} \frac{ |\vc{m}|^{2} }{\vr} + c_v \vr^\gamma \exp \left( \frac{S}{c_v \vr} \right) \right] (T-) } \\
&\leq \liminf_{t \to T-} 
\intO{ \left< \nu_{t,x} ;  \frac{1}{2} \frac{ |\vc{m}|^{2} }{\vr} + c_v \vr^\gamma \exp \left( \frac{S}{c_v \vr} \right)  \right> }
\leq E_0.
\end{split}
\]
The function 
\[
[\vr, S] \mapsto \vr Z \left( \frac{S}{\vr} \right) 
\]
is concave such that
\[
\intO{ \vr Z \left( \frac{S}{\vr} \right)(T-) \varphi } \geq 
\limsup_{t \to T-} \intO{ \left< \mathcal{V}_{t,x}; \vr Z \left( \frac{S}{\vr} \right) \right> \varphi }.
\]
So, \eqref{D4} and \eqref{D5} have the concatenation property as well.
\end{proof}

\begin{Lemma}\label{LDS3}
The mapping
\[
\mathcal{U}:( [\vr_0, \vm_0,S_0], E_0)\in\mathcal D_{\rm Euler} \mapsto \mathcal{U}( [\vr_0, \vm_0,S_0], E_0) \in 2^{\mathcal{T}_{\rm Euler}}
\]
is Borel measurable; where $2^{\mathcal{T}_{\rm Euler}}$ is endowed with the Hausdorff topology defined on compact subsets of the metric space
$L^1_{{\rm loc}}([0, \infty); \mathcal{S}_{\rm Euler})$.
\end{Lemma}
\begin{proof}
As a consequence of the sequential stability from Propositions \ref{ESP2} the claim follows from \cite[Thm. 12.1.8]{StrVar}.
\end{proof}

\begin{proof}[Proof of Theorem \ref{ST1}]
In view of Proposition \ref{dDP1}, the first claim of Theorem \ref{ST1} - the existence of the semiflow $\mathbb{U}$ - follows now from Lemmas~\ref{LDS0}--\ref{LDS3}. Let us finally explain why $\mathbb{U}([ \vr_0, \vm_0,S_0], E_0 )$ can be selected to be maximal in the sense of Definition \ref{DD2}. As stated in Proposition \ref{dDP1} the selection can be chosen to satisfy
\[
\int_0^\infty \exp(- t) \beta (\mathbb{U}([\vr_0, \vm_0,S_0], E_0) (t)) \ \dt
\leq \int_0^\infty \exp(- t) \beta (([\tilde\vr, \tilde{\vc{m}},\tilde S],E_0)(t)) \ \dt.
\]
for any $([\tilde\vr, \tilde{\vc{m}},\tilde S],E_0) \in \mathcal{U}([\vr_0, \vm_0,S_0], E_0)$.
Here $\beta : \mathcal{S}_{\rm Euler} \to R$ is a bounded and continuous function. 
We suppose that 
\begin{align*}
\beta(\vr,\vm,S,E)=\alpha\left(\int_{\Omega}S\,\dif x\right),
\end{align*} 
where $\alpha:R\rightarrow R$ is smooth, bounded and strictly decreasing.
We proceed by contradiction.
Let $([\tvr, \tvm,\tilde S],{E}_{0}) \in \mathcal{U}([\vr_0, \vm_0,S_0], E_0)$ be such that
$
([\tvr, \tvm,\tilde S], {E}_{0}) \succ ([\vr, \vm, S],E_{0}),
$ that is,
$
\int_\Omega\tilde{S}\dx \geq \int_\Omega S\dx
$ in $(0,\infty),$ where we denote $([\vr, \vm, S],E_{0})=\mathbb U([\vr_0, \vm_0,S_0], E_0)$.
Then we get
\[
\alpha\left(\int_\Omega\tilde{S}\dx\right) \leq \alpha\left(\int_\Omega{S}\dx\right) \ \mbox{and}\ \int_0^\infty \exp(-t) \left[ \alpha\left(\int_\Omega{S}\dx\right)-\alpha\left(\int_\Omega\tilde{S}\dx\right)  \right] \dt \leq 0;
\]
whence $\int_\Omega S\dx = \int_\Omega\tilde S\dx$ a.a. in $(0,\infty)$ since $\alpha$ is strictly decreasing. Hence $\mathbb{U}([\vr_0, \vm_0,S_0], E_0)$ is maximal with respect to $\succ$. 
The proof of Theorem \ref{ST1} is hereby complete.
\end{proof}

\section{Concluding discussion}
\label{D}

We have shown the existence of a semiflow solution to the complete Euler system in the class of dissipative solutions. The 
semiflow solution has been selected among other solutions starting from the same initial data. The major issue is, of course, 
uniqueness of such a selection. In accordance with Proposition \ref{DPP2} the selected solution is unique and coincides with 
the strong solutions emanating from the same initial state as long as the latter exists. There is a special class of strong solutions - the equilibrium states. As we shall see below, they enjoy certain stability properties thanks to the fact that the selected solutions
produce the maximal amount of entropy.  

\subsection{Stability of equilibrium states}

The equilibrium states are stationary solutions to the Euler system \eqref{I1}. Specifically, the momentum vanishes 
identically $\vm \equiv 0$, while the density $\vr = \Ov{\vr} > 0$ and the entropy $S = \Ov{S}$ are constant. Note that 
$\Ov{\vr}$ is uniquely determined by the total mass 
\[
M = \Ov{\vr} |\Omega| = \intO{ \vr(t,x) } = \intO{ \vr_0(x) }\ \mbox{for any}\ t > 0.  
\]
Given $\Ov{\vr}$ and the total energy $E_0$, the (constant) equilibrium entropy $\Ov{S}$ is the unique maximizer of the total entropy 
among all states with given mass and energy:
\[ 
\intO{ \Ov{S} } = \sup \left\{ \intO{ S } \ \Big|\, \vr, S \in L^1(\Omega),\, \vr \geq 0,
\, \intO{ c_v {\vr}^\gamma \exp \left( \frac{S}{c_v {\vr}} \right) } = E_0,\ \intO{ \vr } = M \right\}.
\]
Indeed, suppose that $\Ov{\vr}$, $\Ov{S}$ are constant, and 
\[
\intO{ \Ov{\vr} } = \intO{\vr} = M,\ \intO{ \Ov{S} } \leq \intO{S},\ 
\intO{ \Ov{\vr} e(\Ov{\vr}, \Ov{S}) } = \intO{ \vr e(\vr, S) } = E_0.
\]
Normalizing the above relation by a factor $\frac{1}{|\Omega|}$ we may suppose $|\Omega| = 1$. Using convexity of the 
function $\Phi (\vr, S) \equiv \vr e(\vr, S)$ (cf. Lemma \ref{LC1}) and Jensen's inequality, we get 
\[
E_0 = \intO{ \Phi(\Ov{\vr}, \Ov{S}) } = \Phi \left( \intO{ \Ov{\vr} }, \intO{\Ov{S}} \right) \leq
\Phi \left( \intO{ {\vr} }, \intO{{S}} \right) \leq \intO{ \Phi (\vr, S) } = E_0. 
\]
As $\Phi$ is strictly increasing in $S$ and strictly convex in $(\vr,S)$, we get successively $\intO{S} = \intO{\Ov{S}}$, and 
$\vr = \Ov{\vr}$, $S = \Ov{S}$.

We claim that a \emph{maximal solution} can only see the equilibria with the energy $E_0$. 

\begin{Proposition}[Stability of equilibria] \label{Padm}

Let $([\vr, \vm, S], E_0)$ be a dissipative solution to the Euler system emanating from the initial data 
\[
\vr_0 = \Ov{\vr} - \ \mbox{a positive constant},\ \vm_0 = 0,\ S_0  - \mbox{a constant}. 
\]
Suppose that $([\vr, \vm, S], E_0)$ is maximal in the sense of Definition \ref{DD2}.

Then 
\begin{equation} \label{max}
\vr = \Ov{\vr},\ \vm = 0,\ S = \Ov{S},\ \mbox{where} \ \intO{ c_v \Ov{\vr}^\gamma \exp \left( \frac{S}{c_v \Ov{\vr}} \right) }
= \Ov{S}.
\end{equation}

\end{Proposition}

\begin{proof}

As the equilibrium solution maximizes the total entropy, we have $\Ov{S} \geq S_0$; whence \eqref{max} is a dissipative solution. 
By the same token, it is a maximal solution and any other solution ``larger'' in the sense of $\succ$ has the same entropy, meaning it concides with the solution \eqref{max}.
\end{proof}

\subsection{Dafermos' admissibility criteria}\label{sec:D}

Finally, we discuss some implications of Dafermos' admissibility criteria introduced in Section \ref{AD}. 
As we have seen in the proof of Theorem~\ref{ST1}, maximimality of the entropy production by dissipative solutions has been enforced by maximizing the integral 
\[
\int_0^\infty \exp(-t) \intO{ S(t,x) } \dt  
\]
modulo the cut--off function $\alpha$. Actually the same result could have been achieved by maximizing 
\begin{align}\label{eq:1403b}
\int_0^\infty \exp(- \lambda t) \intO{ S(t,x) } \dt,\ \lambda > 0.  
\end{align}
Thus, a natural question arises, namely, what would be the resulting solution if we maximized successively the entropy 
for a family $\lambda_n \to \infty$.
The following result implies: if there is a solution such that it has maximal Laplace transform of the total entropy evaluated at every $\lambda_{n}$ from a sequence $\lambda_{n}\to\infty$, then it satisfies the entropy production criterium by Dafermos in the form \eqref{eq:D}, that is, it is maximal with respect to  the partial ordering $\succ_{\mathcal{D}}$.\\
With this motivation in mind, we employ the following notation for  $F,G\in L^{\infty} (0,\infty)\cap \BV_{\rm loc}([0,\infty))$: We say  that $F\succ_{\mathcal D}G$ provided there exists $\delta>0$ such that $F(t+)\geqslant G(t+)$ for all $t\in (0,\delta)$; and we say that $F\succ_{\mathcal F}G$ provided there exists a sequence $(\tau_{n})_{n}$, $\tau_{n}>0$, $\tau_{n}\to 0$  such that $F(\tau_{n}+)\geqslant G(\tau_{n}+)$.

\begin{Lemma}\label{lem:10}
Let $\mathcal{G}\subset L^{\infty}(0,\infty)\cap \BV_{\rm loc}([0,\infty))$ and assume  that there exists $F \in \mathcal{G}$ such that for some sequence $\lambda_{n}\to\infty$ it holds
  \[ \lambda_{n} \int_0^{\infty} \exp(- \lambda_{n} t) F (t)\, \mathd t \geqslant \lambda_{n}
     \int_0^{\infty} \exp(- \lambda_{n} t) G (t) \,\mathd t \qquad \tmop{for}
     \tmop{ all }  G \in \mathcal{G} \text{ and } n\in\mathbb{N}. \]
  Then for every $G\in\mathcal{G}\setminus\{F\}$ one of the following cases holds: either 
  $ F \succ_{\mathcal{D}} G$
   or $F \sim_{\mathcal{F}} G$, meaning, in the latter case it holds simultaneously $F \succ_{\mathcal{F}} G$ and $G \succ_{\mathcal{F}} F$. Accordingly, $F$ is maximal with respect to $\succ_{\mathcal{D}}$.
%
%
%
\end{Lemma}


\begin{proof}
Fix an arbitrary  $G\in\mathcal{G}\setminus\{F\}$ and assume that the first claim $ F \succ_{\mathcal{D}} G$ is not valid. That is, for every $\delta>0$ there exists $t\in (0,\delta)$ such that $F(t+)< G(t+)$ This directly implies $G\succ_{\mathcal{F}}F$. So it remains to show the converse statement, i.e.  $F\succ_{\mathcal{F}}G$.

{\em Case 1:} If there is a $\delta>0$  such that $F(t+)<G(t+)$ for all $t\in (0,\delta)$ then we write
  \[ \lambda_{n} \int_0^{\infty} \exp(- \lambda_{n} t)  [G (t) - F (t)] \,\mathd t 
  \]
  \[=
     \lambda_{n} \int_0^{\delta} \exp(- \lambda_{n} t)  [G (t) - F (t)] \,\mathd t +
     \lambda_{n} \int_{\delta}^{\infty} \exp(- \lambda_{n} t)  [G (t) - F (t)] \,\mathd t
  \]
  \[ \geqslant \lambda_{n} \exp(- \lambda_{n} \delta) \int_0^{\delta}  [G (t) - F (t)]
     \,\mathd t - (\| F \|_{L^{\infty}} + \| G \|_{L^{\infty}}) \exp(- \lambda_{n}
     \delta), \]
  where the right hand side is strictly positive provided $n$ was chosen
  sufficiently large. This is a contradiction with the maximality of the corresponding Laplace transforms of $F$.
  
{\em Case 2:} If that is not the case, then the functions $F,G$ oscillate around each other in the sense that for every $\delta>0$ there exists $t\in (0,\delta)$ such that $F(t+)\geqslant G(t+)$. Hence $F\succ_{\mathcal{F}}G$ and the second claim, namely, $F\sim_{\mathcal{F}}G$, is valid.

Finally, as already observed in Section \ref{AD}, maximality with respect to $\succ_{\mathcal F}$ implies maximality with respect to $\succ_{\mathcal D}$. Thus, $F$ is maximal with respect to $\succ_{\mathcal{D}}$ and the proof is complete.
\end{proof}

In other words, if for some initial data the Euler system \eqref{I1}--\eqref{I4} possesses a solution whose total entropy has maximal Laplace transforms evaluated at a sequence $\lambda_{n}\to\infty$, then the Dafermos' criterium is satisfied. Note that since we have a semigroup it is enough to test the  criterium at the time $t=0$.

We now argue that such a solution can be always chosen by our selection process considering a suitable order of minimizers in the 
procedure described in \cite[Section 5.1]{BreFeiHof19}. The latter one considers bounded continuous functionals $\beta:\mathcal S\rightarrow R$ on the phase space. The final selection $\mathbb U$ maximizes (or minimizes) $\beta(\mathbb U)$ pointwise in time
by selecting the maximizer (or minimizer) of the functionals
\begin{align}\label{eq:1403}
I_n[\mathbb V]=\int_0^\infty \exp(-\lambda_{n}t) \beta(\mathbb V) \dt,\ \mathbb V\in \BV_{\rm loc}([0,\infty);\mathcal S),
\end{align}
from the solution set.
In \cite[Section 5.1]{BreFeiHof19}, the sequence $(\lambda_n)_{n}$ is chosen to be dense in $(0,\infty)$. However, the density of $(\lambda_n)_{n}$ is only needed to apply Lerch's theorem implying the uniqueness of the Laplace transform. In fact, it is sufficient that the family of functionals $I_n$ in \eqref{eq:1403} separates points, which can be achieved by choosing a suitable increasing sequence $\lambda_{n} \to\infty$.
This follows from the following generalization of classical Lerch's theorem. Consequently, it is enough that the total entropy of our selection maximizes \eqref{eq:1403b} for $\lambda=\lambda_n$, $n\in \mathbb{N}$.

\begin{Lemma}
  Let $\lambda > 0$ and $(\zeta_n)_{n} \subset (0, \infty)$
  be a sequence with the following property: for all $n, m \in \mathbb{N}$
  there exists $k \in \mathbb{N}$ such that $\zeta_n + \zeta_m =
  \zeta_k$. Then the set of functionals
  \[ L^{\infty} (0, \infty) \rightarrow \mathbb{R}, \quad F \mapsto
     \int_0^{\infty} \exp(- (\lambda + \zeta_n) t) F (t) \,\mathd t, \qquad n
     \in \mathbb{N}, \]
  separates points.
\end{Lemma}

\begin{proof}
  Due to the condition on the sequence $(\zeta_n)_n$, the set of finite
  linear combinations of functions $(e^{- \zeta_n t})_n$ forms a subalgebra
  of $C_0 ([0, \infty))$ which separates points and vanishes nowhere. Hence by
  the locally compact version of the Stone--Weierstrass theorem it is dense in
  $C_0 ([0, \infty))$ and consequently dense in $L^1 (0, \infty)$. The claim
  now follows since $L^1 (0, \infty)$ is the predual of $L^{\infty} (0,
  \infty)$: let $F$ be such that all the functionals above are zero. Then for
  every $g \in L^1 (0, \infty)$
  \[ \int_0^{\infty} g (t) \exp(- \lambda t) F (t) \,\mathd t = 0 \]
  and consequently $e^{- \lambda t} F (t) = 0$ for a.e. $t \in (0, \infty)$
  hence $F = 0$.
\end{proof}

From the above Lemma we see that a minimal sequence $(\lambda_{n})_{n}$ needed for Lerch's theorem is an arithmetic sequence of the form $\lambda_{n}=\lambda+n\zeta$ where $\lambda,\zeta>0$.
For completeness we note that the converse implication regarding the Dafermos' criterium   is not valid. More precisely, if a function satisfies the Dafermos' criterium \eqref{eq:D} then it will not necessarily be chosen by our selection. This observation is based on the following result.

\begin{Lemma}
 Let $\mathcal{G}\subset L^{\infty}(0,\infty)\cap \BV_{\rm loc}([0,\infty))$ and assume that $F\in \mathcal{G}$ satisfies $ F\succ_{\mathcal{D}}  G$
  for every $G\in\mathcal{G}$. 
  Then there exists $\lambda_0=\lambda_{0}(F,G) > 0$ such that for all $\lambda \geqslant
  \lambda_0$ it holds
  \[ \lambda \int_0^{\infty} \exp(- \lambda t) F (t) \,\mathd t > \lambda
     \int_0^{\infty} \exp(- \lambda t) G (t) \,\mathd t. \]
\end{Lemma}

\begin{proof}
  We write
  \[ \lambda \int_0^{\infty} \exp(- \lambda t)  [F (t) - G (t)] \,\mathd t =
     \lambda \int_0^{\delta} \exp(- \lambda t)  [F (t) - G (t)] \,\mathd t +
     \lambda \int_{\delta}^{\infty} \exp(- \lambda t)  [F (t) - G (t)] \,\mathd t
  \]
  \[ \geqslant \lambda \exp(- \lambda \delta) \int_0^{\delta}  [F (t) - G (t)]
    \, \mathd t - (\| F \|_{L^{\infty}} + \| G \|_{L^{\infty}}) \exp(- \lambda
     \delta). \]
Hence there exists $\lambda_0 > 0$ (depending on $F, G$) such that the right
  hand side is strictly positive provided $\lambda \geqslant \lambda_0$.
\end{proof}

In other words, one would like to start the selection procedure given by an increasing  sequence $(\lambda_{n})_{n}$ with $\lambda_{1}$ as large as possible, in order to guarantee that the solution satisfying the Dafermos' criterium is selected. However, a uniform choice of $\lambda_{1}$ is not a priori known at the moment.

\def\cprime{$'$} \def\ocirc#1{\ifmmode\setbox0=\hbox{$#1$}\dimen0=\ht0
  \advance\dimen0 by1pt\rlap{\hbox to\wd0{\hss\raise\dimen0
  \hbox{\hskip.2em$\scriptscriptstyle\circ$}\hss}}#1\else {\accent"17 #1}\fi}


\end{document}